\documentclass[12pt,a4paper]{article}
\usepackage{amsmath}
\usepackage{amssymb}
\usepackage{amsthm}
\usepackage{latexsym}

\usepackage{mathtools}
\usepackage{bm}
\usepackage{upref} 

\usepackage{fullpage}

\allowdisplaybreaks

\theoremstyle{definition}
\newtheorem{theorem}{Theorem}[section]
\newtheorem{proposition}[theorem]{Proposition}

\newtheorem{lemma}[theorem]{Lemma}

\newtheorem*{claim*}{Claim}

\newtheorem{remark}{Remark}

\numberwithin{remark}{section}

\numberwithin{theorem}{section}

\numberwithin{equation}{section}

\newcommand\C{{\mathbb C}}

\newcommand\R{{\mathbb R}}

\newcommand\Z{{\mathbb Z}}

\newcommand\cF{{\mathcal F}}

\newcommand\cP{{\mathcal P}}

\newcommand\ve{\varepsilon}
\newcommand\vp{\varphi}
\renewcommand\d{\partial}
\newcommand\vD{\lvert\nabla\rvert}
\newcommand\aD{\langle\nabla\rangle}

\DeclareMathOperator{\re}{Re}

\DeclareMathOperator{\supp}{supp}

\newcommand\dsp{\displaystyle}

\newcommand{\ol}[1]{\overline{#1}}
\newcommand{\wh}[1]{\widehat{#1}}
\newcommand{\wt}[1]{\widetilde{#1}}

\newcommand{\abs}[1]{\lvert#1\rvert}
\newcommand{\Abs}[1]{\bigl\lvert#1\bigr\rvert}
\newcommand{\ABS}[1]{\Bigl\lvert#1\Bigr\rvert}

\newcommand{\norm}[1]{\lVert#1\rVert}
\newcommand{\Norm}[1]{\bigl\lVert#1\bigr\rVert}
\newcommand{\NM}[1]{\Bigl\lVert#1\Bigr\rVert}

\newcommand{\rb}[1]{(#1)}
\newcommand{\Rb}[1]{\bigl(#1\bigr)}
\newcommand{\RB}[1]{\Bigl(#1\Bigr)}

\newcommand{\cb}[1]{\{#1\}}
\newcommand{\Cb}[1]{\bigl\{#1\bigr\}}

\newcommand{\ab}[1]{\langle#1\rangle}
\newcommand{\Ab}[1]{\bigl\langle#1\bigr\rangle}

\newcommand{\sqb}[1]{[#1]}
\newcommand{\Sb}[1]{\bigl[#1\bigr]}
\newcommand{\SB}[1]{\Bigl[#1\Bigr]}

\title{Small energy scattering for radial solutions to the generalized Zakharov system}

\author{Jun Kato \and Osamu Tojo}

\date{}

\begin{document}

\maketitle

\begin{abstract}
We prove the small energy scattering for the three-dimensional generalized Zakharov system with radial symmetry based on the idea by Guo and Nakanishi \cite{guo-nakanishi_1}, which treats the usual Zakharov system.
For the proof, we use the frequency-localized normal form reduction, and the radially improved Strichartz estimates.
The relation between the solution to the integral equations, which includes the unusual boundary terms, and the original differential equations is also considered.
\end{abstract}

\section{Introduction}

In this paper, we consider the following generalized Zakharov system in three spatial dimensions,
\begin{equation*}
(Z_{\gamma})\quad \left\{
\begin{aligned}
& i\partial_{t}u -\Delta u=nu,\quad (t,x)\in (0,\infty)\times\R^{3},\\
& \partial_{t}^{2} n-\Delta n=\abs{\nabla}^{1+\gamma} \abs{u}^{2},\quad (t,x)\in (0,\infty)\times\R^{3},\\
& u|_{t=0}=u_{0},\ n|_{t=0}=n_{0},\ \partial_{t}n|_{t=0}=n_{1},\quad x\in\R^{3},
\end{aligned}
\right.
\end{equation*}
where $\gamma\in [-1,1]$,
\[
  u:[0,\infty)\times\R^{3}\to \C,\qquad
  n:[0,\infty)\times\R^{3}\to\R.
\]
We notice that the case $\gamma=1$ is the Zakharov system, modeling propagation of Langmuir waves in an ionized plasma, and the case $\gamma=-1$ is the (mass-less) Yukawa system, which is a model for the interaction between a meson and a nucleon.

The system $(Z_{\gamma})$ is interesting as one of a source of the Hartree type equation.
In fact, if we consider the second equation with the parameter
\[
  \frac{1}{\alpha^{2}}\,\partial_{t}^{2} n-\Delta n=\abs{\nabla}^{1+\gamma} \abs{u}^{2},
\] 
then taking the subsonic limit ($\alpha\to\infty$) we obtain 
\[
  i\partial_{t}u -\Delta u=
  c_{\gamma} \rb{\abs{x}^{-2-\gamma}\ast \abs{u}^{2}}u.
\]
Our interest to this system is to know how the size of $\gamma$ effects to the global well-posedness to the system for small data.

The system $(Z_{\gamma})$ has conserved quantities $M=\norm{u(t)}_{L^{2}}$, and the energy
\begin{align*}
  E=\int_{\R^{3}}\Cb{ \abs{\nabla u(t)}^{2} + \frac{1}{2}\,\Abs{\abs{\nabla}^{\frac{1-\gamma}{2}-1}\partial_{t}n(t)}^{2} +\frac{1}{2}\,\Abs{\abs{\nabla}^{\frac{1-\gamma}{2}} n(t)}^{2} -n(t)\abs{u(t)}^{2} }\,dx.
\end{align*}
Below we focus on the solution of the energy class
\[
  u\in C\rb{[0,\infty);H^{1}},\quad n\in C\rb{[0,\infty);H^{\frac{1-\gamma}{2}}},\quad
  \d_{t}n \in C\rb{[0,\infty);H^{-\frac{1+\gamma}{2}}}.
\]

As for the Zakharov system $(Z_{1})$, the existence of unique global solution and scattering is proved by Hani, Pusateri, and Shatah \cite{hani-pusateri-shatah} for small initial data which belongs to some weighted Sobolev space.
Their method is based on the space-time resonance method, which is developed by German, Masmoudi, and Shatah. See e.g. \cite{germain-masmoudi-shatah}.
Guo and Nakanishi \cite{guo-nakanishi_1} is proved the existence of unique radial global solution and scattering in the framework of the energy class for small initial data.
Their method is based on the use of radially improved Strichartz estimates, after applying the method of the normal form partially.
In that way, they were able to avoid the use of weights and were able to close the argument in the energy class.

As for the Zakharov system $(Z_{\gamma})$ with $\gamma >0$, based on the space-time resonance method, Beck, Pusateri, Sosoe, and Wong \cite{beck-pusateri-sosoe-wong} proved the existence of unique global solution and scattering for small initial data which belongs to some weighted Sobolev space.

The aim of this paper is to study the condition on $\gamma$ which ensures the existence of unique global solution to $(Z_{\gamma})$ in the framework of the energy class under the assumption of radial symmetry, based on the idea of Guo and Nakanishi \cite{guo-nakanishi_1}.

This paper is organized as follows.
In the rest of this section, after explaining notation of frequency localization, we explain the derivation of the integral equation which we consider, based on the idea of Guo and Nakanishi \cite{guo-nakanishi_1}.
Then we describe our main results.
In section 2, we summarize basic estimates, especially the radially improved Strichartz estimates, and the estimate of the bilinear Fourier multiplier.
In section 3, we summarize the estimates of nonlinear terms.
In sections 4 and 5, we give proof of Theorems 1.1 and 1.2, respectively.
In section \ref{sec-appendix}, we give a proof of the scattering in the case $\gamma=1$, since it is just mentioned to hold in \cite{guo-nakanishi_1}, and we need a lemma additionally.

\subsection{Reduction of the system}

We first reduce the system to the first order system by setting $\dsp N=n-i\abs{\nabla}^{-1}n_{t}$,
\begin{equation}\label{differential-equation}
\left\{
\begin{aligned}
& i\partial_{t}u -\Delta u=\frac{1}{2}(Nu+\overline{N}u),\quad (t,x)\in (0,\infty)\times\R^{3},\\
& i\partial_{t} N+\abs{\nabla} N=\abs{\nabla}^{\gamma} |u|^{2}.
\end{aligned}\right.
\end{equation}
Note that $n=\re N$.
Then, denoting
\[
  S(t)=e^{-it\Delta},\quad
  W(t)=e^{it\abs{\nabla}},
\]
the corresponding integral equations are the following.
\begin{equation}\label{integral-equation}
\left\{
\begin{aligned}
  & u(t)=S(t)u_{0}-\frac{i}{2}\int_{0}^{t}S(t-\tau)\Rb{N(\tau)+\ol{N}(\tau)}u(\tau)\,d\tau,\\
  & N(t)=W(t)N_{0} -i\int_{0}^{t}W(t-\tau)\abs{\nabla}^{\gamma}\abs{u(\tau)}^{2}\,d\tau.
\end{aligned}\right.
\end{equation}
In our argument, the term $\ol{N}u$ makes no essential difference from $Nu$, and hence for simplicity, we assume the nonlinear term in the first equation $Nu$ below.

Following Guo-Nakanishi \cite{guo-nakanishi_1}, we apply the method of the normal form after decomposing the nonlinear terms in terms of the frequency.
We give its notation in next subsection.


\subsection{Notation}\label{notation}

Let $\varphi\in C_{0}^{\infty}(\R^{n})$ be radial and satisfy $0\le \varphi (\xi)\le 1$,
\[
  \varphi (\xi)=\left\{ 
  \begin{aligned} & 1,\quad |\xi|\le 5/4,\\
    & 0,\quad |\xi| >8/5,
  \end{aligned}\right.
\]
and we set $\phi(\xi)=\varphi(\xi)-\varphi(2\xi)$.
For $j\in\Z$, we define $\dsp \phi_{j}(\xi)=\phi\rb{\xi/2^{j}}$, and $\dsp \varphi_{j}(\xi)=\varphi\rb{\xi/2^{j}}$.
Then, $\cb{\phi_{j}}_{j\in\Z}\in C_{0}^{\infty}(\R^{n})$ satisfy the following.
\begin{align}
  & \bullet\quad  \supp\phi_{j}\subset \bigl\{\xi\in\R^{n}\mid \frac{5}{8}\,2^{j} \le |\xi| \le \frac{8}{5}\, 2^{j}\bigr\},\quad j\in\Z,  \label{phi-j-support-cond}\\ 
  & \bullet\quad \varphi(\xi)+\sum_{j=1}^{\infty} \phi_{j}(\xi)=1,\quad \xi\in\R^{n},
    \notag\\
  & \bullet\quad \sum_{j\in \Z} \phi_{j}(\xi)=1,\quad \xi\in\R^{n}\setminus \{0\},
    \notag\\
  & \bullet\quad \phi_{j}\phi_{k}=0 \quad\text{if}\quad |j-k|\ge 2.
    \notag
\end{align}
Now we define
\[
  P_{j}f=\cF^{-1}\Sb{\phi_{j}\wh{f}\,},\qquad 
  P_{\le j}f=\cF^{-1}\Sb{\varphi_{j}\wh{f}\,},\qquad j\in\Z.
\]
For pair of functions $f$, $g$, we define
\[
  (fg)_{LH}=\sum_{j\in\Z} P_{\le j-4}f\, P_{j}g,\quad
  (fg)_{HL}=\sum_{j\in\Z} P_{j}f\, P_{\le j-4}g,\quad
  (fg)_{HH}=\sum_{\abs{j-k}\le 3}P_{j}f\, P_{k}g.
\]
so that
\begin{equation}\label{paraproduct-decomp}
  fg=(fg)_{LH}+(fg)_{HL}+(fg)_{HH}.
\end{equation}
To distinguish the resonant interaction, we also use
\[
  (fg)_{RL}=\sum_{\abs{j}\le 2}P_{j}f\, P_{\le j-4}g,\quad
  (fg)_{XL}=\sum_{\abs{j}\ge 3}P_{j}f\, P_{\le j-4}g
\]
so that
\begin{equation}\label{resonant-decomp}
  (fg)_{HL}=(fg)_{RL} + (fg)_{XL},
\end{equation}
and similarly $(fg)_{LR}$, $(fg)_{LX}$.
When we regard these expressions as the bilinear Fourier multipliers, we write the symbols as follows,
\begin{align*}
  (fg)_{LH}&=\cF^{-1}\SB{(2\pi)^{-\frac{3}{2}} \int \cP_{LH}(\xi-\eta,\eta)\wh{f}(\xi-\eta)\,\wh{g}(\eta)\,d\eta}\\
  &=(2\pi)^{-3}\iint e^{ix\cdot (\xi +\eta)}\,
    \cP_{LH}(\xi,\eta)\wh{f}(\xi)\,\wh{g}(\eta)\,d\xi d\eta,
\end{align*}
where
\[
  \cP_{LH}(\xi,\eta)=\sum_{j\in\Z}\vp_{\le j-4}(\xi)\phi_{j}(\eta).
\]
We also use the notation like
\[
  (fg)_{LH+HL}=(fg)_{LH}+(fg)_{HL}.
\]

\subsection{The Schr\"odinger part}

Applying the Fourier transform to the integral equation, we have
\begin{align*}
  \widehat{u}(t,\xi)& = e^{it|\xi|^{2}}\widehat{u}_{0}(\xi)
    -i\int_{0}^{t}e^{i(t-\tau)|\xi|^{2}}\mathcal{F}[Nu](\tau,\xi)\,d\tau\\
  &=e^{it|\xi|^{2}}\widehat{u}_{0}(\xi)
    -i\int_{0}^{t}e^{i(t-\tau)|\xi|^{2}}\mathcal{F}[(Nu)_{XL}](\tau,\xi)\,d\tau\\
  &\qquad \qquad {}-i\int_{0}^{t}e^{i(t-\tau)|\xi|^{2}}\mathcal{F}[(Nu)_{RL+LH+HH}](\tau,\xi)\,d\tau.
\end{align*}
Then, for the second term on the right hand side, we first rewrite as follows.
\begin{align*}
  &\int_{0}^{t}e^{i(t-\tau)|\xi|^{2}}\mathcal{F}[(Nu)_{XL}](\tau,\xi)\,d\tau\\
  &=(2\pi)^{-\frac{3}{2}}\int_{0}^{t}\int e^{i(t-\tau)|\xi|^{2}}\mathcal{P}_{XL}(\xi-\eta,\eta)\wh{N}(\tau,\xi-\eta)\wh{u}(\tau,\eta)\,d\tau d\eta\\
  &=(2\pi)^{-\frac{3}{2}} e^{it \abs{\xi}^{2}} \int_{0}^{t}\int e^{i\tau \omega(\xi-\eta,\eta)}\mathcal{P}_{XL}(\xi-\eta,\eta)e^{-i\tau \abs{\xi-\eta}}\wh{N}(\tau,\xi-\eta)e^{-i\tau \abs{\eta}^{2}}\wh{u}(\tau,\eta)\,d\tau d\eta,
\end{align*}
where
\[
  \omega (\xi,\eta)=-\abs{\xi+\eta}^{2} + \abs{\xi} + \abs{\eta}^{2}.
\]
Then, we apply the integration by parts in the time variable by using the equations
\begin{equation}\label{fourier-differential-equation}
\begin{aligned}
  &e^{i\tau\omega}=\frac{1}{i\,\omega}\partial_{\tau}e^{i\tau\omega},\qquad
  \d_{\tau}\cb{e^{-i\tau\abs{\xi}^{2}}\wh{u}(\tau,\xi)}=-i\, e^{-i\tau\abs{\xi}^{2}} \cF\sqb{Nu}(\xi),\\
  &\d_{\tau}\cb{e^{-i\tau\abs{\xi}}\wh{N}(\tau,\xi)}=-i\, e^{-i\tau\abs{\xi}}\abs{\xi}^{\gamma} \cF\sqb{u\,\ol{u}}(\xi)
\end{aligned}
\end{equation}
to obtain
\begin{align*}
  &\int_{0}^{t}\int \frac{1}{i\,\omega(\xi-\eta,\eta)}\,\cb{\d_{\tau}e^{i\tau \omega(\xi-\eta,\eta)}}\mathcal{P}_{XL}(\xi-\eta,\eta)e^{-i\tau \abs{\xi-\eta}}\wh{N}(\tau,\xi-\eta)e^{-i\tau \abs{\eta}^{2}}\wh{u}(\tau,\eta)\,d\tau d\eta\\
  &=e^{-it\abs{\xi}^{2}} \int \frac{\cP_{XL}(\xi-\eta,\eta)}{i\, \omega(\xi-\eta,\eta)}\, \wh{N}(t,\xi-\eta)\wh{u}(t,\eta)\,d\eta
    -\int \frac{\cP_{XL}(\xi-\eta,\eta)}{i\, \omega(\xi-\eta,\eta)}\, \wh{N}(0,\xi-\eta)\wh{u}(0,\eta)\,d\eta\\
  &\qquad {}+i\int_{0}^{t}\int e^{i\tau \omega(\xi-\eta,\eta)}\, \frac{\mathcal{P}_{XL}(\xi-\eta,\eta)}{i\,\omega(\xi-\eta,\eta)}\, e^{-i\tau \abs{\xi-\eta}} \abs{\xi-\eta}^{\gamma} \cF\sqb{u\,\ol{u}}(\tau,\xi-\eta)\,e^{-i\tau \abs{\eta}^{2}}\wh{u}(\tau,\eta)\,d\tau d\eta\\
  &\qquad {}+i\int_{0}^{t}\int e^{i\tau \omega(\xi-\eta,\eta)}\, \frac{\mathcal{P}_{XL}(\xi-\eta,\eta)}{i\,\omega(\xi-\eta,\eta)}\, e^{-i\tau \abs{\xi-\eta}}\wh{N}(\tau,\xi-\eta)\,e^{-i\tau \abs{\eta}^{2}} \cF\sqb{Nu}(\tau,\eta)\,d\tau d\eta\\
  &=(2\pi)^{\frac{3}{2}}\Cb{-i\, e^{-it\abs{\xi}^{2}}\cF\Omega(N(t),u(t)) +i \cF\Omega(N(0),u(0))\\
  &\qquad {} + \int_{0}^{t} e^{-i\tau\abs{\xi}^{2}} \cF\Omega\Rb{\vD^{\gamma}\rb{u\,\ol{u}}, u}(\tau,\xi)\,d\tau
    + \int_{0}^{t} e^{-i\tau\abs{\xi}^{2}} \cF\Omega(N, Nu)(\tau,\xi)\,d\tau},
\end{align*}
where
\begin{equation}\label{bilinear-op-omega}
\begin{aligned}
  \Omega(f,g)(x)&=(2\pi)^{-\frac{3}{2}}\cF^{-1}\int  \frac{\cP_{XL}(\xi-\eta,\eta)}{\omega(\xi-\eta,\eta)}\, \wh{f}(\xi-\eta)\,\wh{g}(\eta)\,d\eta\\
  &=(2\pi)^{-3}\iint e^{ix\cdot(\xi+\eta)}\,\frac{\cP_{XL}(\xi,\eta)}{\omega(\xi,\eta)}\, \wh{f}(\xi)\,\wh{g}(\eta)\,d\xi d\eta.
\end{aligned}
\end{equation}
Here, we notice that on the support of $\cP_{XL}$,
\[
  \omega(\xi,\eta)=-\abs{\xi+\eta}^{2} + \abs{\xi} + \abs{\eta}^{2}
  \sim -\abs{\xi}^{2} + \abs{\xi},
\]
since $\abs{\eta}\ll \abs{\xi}$, and the resonant frequency $\abs{\xi}\sim 1$ is eliminated as a result of subtraction of the resonant interaction $(Nu)_{RL}$.
See section \ref{subsec-bilinear-FM} for the details.

Finally, we obtain
\begin{equation}\label{int-normal-schrodinger}
\begin{aligned}
  u(t)&=S(t)u_{0}  + S(t)\Omega(N_{0},u_{0}) - \Omega(N(t), u(t))\\
    &\qquad {} -i \int_{0}^{t} S(t-\tau) (Nu)_{RL+LH+HH}(\tau)\,d\tau\\
    &\qquad {} -i \int_{0}^{t} S(t-\tau) \Omega\Rb{\vD^{\gamma}\rb{u\,\ol{u}}, u}(\tau)\,d\tau
   -i \int_{0}^{t} S(t-\tau) \Omega(N, Nu)(\tau)\,d\tau.
\end{aligned}
\end{equation}

\subsection{The wave part}

Applying the Fourier transform to the integral equation, we have
\begin{align*}
  \widehat{N}(t,\xi)& = e^{it|\xi|}\widehat{N}_{0}(\xi)
    -i\int_{0}^{t}e^{i(t-\tau)|\xi|}\abs{\xi}^{\gamma}\mathcal{F}[u\,\ol{u}](\tau,\xi)\,d\tau\\
  &=e^{it|\xi|}\widehat{N}_{0}(\xi)
   -i\int_{0}^{t}e^{i(t-\tau)|\xi|}\abs{\xi}^{\gamma}\mathcal{F}[(u\,\ol{u})_{XL+LX}](\tau,\xi)\,d\tau\\
    &\qquad -i\int_{0}^{t}e^{i(t-\tau)|\xi|}\abs{\xi}^{\gamma}\mathcal{F}[(u\,\ol{u})_{HH+RL+LR}](\tau,\xi)\,d\tau.
\end{align*}
Then, for the second term on the right hand side, we first rewrite as follows.
\begin{align*}
  &\int_{0}^{t}e^{i(t-\tau)|\xi|}\abs{\xi}^{\gamma}\mathcal{F}[(u\,\ol{u})_{XL+LX}](\tau,\xi)\,d\tau\\
  &=(2\pi)^{-\frac{3}{2}} \abs{\xi}^{\gamma} \int_{0}^{t}\int e^{i(t-\tau)|\xi|}\mathcal{P}_{XL+LX}(\xi-\eta,\eta)\wh{u}(\tau,\xi-\eta)\wh{\ol{u}}(\tau,\eta)\,d\tau d\eta\\
  &=(2\pi)^{-\frac{3}{2}}\abs{\xi}^{\gamma} e^{it \abs{\xi}} \int_{0}^{t}\int e^{i\tau \theta(\xi-\eta,\eta)}\mathcal{P}_{XL+LX}(\xi-\eta,\eta)e^{-i\tau \abs{\xi-\eta}^{2}}\wh{u}(\tau,\xi-\eta)\cb{\ol{e^{-i\tau \abs{\eta}^{2}}\wh{u}(\tau,-\eta)}}\,d\tau d\eta,
\end{align*}
where
\[
  \theta (\xi,\eta)=-\abs{\xi+\eta} + \abs{\xi}^{2} - \abs{\eta}^{2}.
\]
Then, we apply the integration by parts in the time variable by using the equations \eqref{fourier-differential-equation} to obtain
\begin{align*}
  &\int_{0}^{t}\int \frac{1}{i\,\theta(\xi-\eta,\eta)}\,\cb{\d_{\tau}e^{i\tau \theta(\xi-\eta,\eta)}}\mathcal{P}_{XL+LX}(\xi-\eta,\eta)e^{-i\tau \abs{\xi-\eta}^{2}}\wh{u}(\tau,\xi-\eta)\cb{\ol{e^{-i\tau \abs{\eta}^{2}}\wh{u}(\tau,-\eta)}}\,d\tau d\eta\\
  &=e^{-it\abs{\xi}} \int \frac{\cP_{XL+LX}(\xi-\eta,\eta)}{i\, \theta(\xi-\eta,\eta)}\, \wh{u}(t,\xi-\eta)\ol{\wh{u}(t,-\eta)}\,d\eta\\
  &\qquad {}  -\int \frac{\cP_{XL+LX}(\xi-\eta,\eta)}{i\, \theta(\xi-\eta,\eta)}\, \wh{u}(0,\xi-\eta)\ol{\wh{u}(0,-\eta)}\,d\eta\\
  &\qquad {}+i\int_{0}^{t}\int e^{i\tau \theta(\xi-\eta,\eta)}\, \frac{\mathcal{P}_{XL+LX}(\xi-\eta,\eta)}{i\,\theta(\xi-\eta,\eta)}\, e^{-i\tau \abs{\xi-\eta}^{2}}  \cF\sqb{Nu}(\tau,\xi-\eta)\,\ol{e^{-i\tau \abs{\eta}^{2}}\wh{u}(\tau,-\eta)}\,d\tau d\eta\\
  &\qquad {}-i\int_{0}^{t}\int e^{i\tau \theta(\xi-\eta,\eta)}\, \frac{\mathcal{P}_{XL+LX}(\xi-\eta,\eta)}{i\,\theta(\xi-\eta,\eta)}\, e^{-i\tau \abs{\xi-\eta}^{2}}\wh{u}(\tau,\xi-\eta)\,\,\ol{e^{-i\tau \abs{\eta}^{2}} \cF\sqb{Nu}(\tau,-\eta)}\,d\tau d\eta\\
  &=(2\pi)^{\frac{3}{2}}\Cb{-i\, e^{-it\abs{\xi}}\cF\Theta(N(t),\ol{u(t)}) +i \cF\Theta(N(0),\ol{u(0)})\\
  &\qquad {} + \int_{0}^{t} e^{-i\tau\abs{\xi}} \cF\Theta\rb{Nu, \ol{u}}(\tau,\xi)\,d\tau
    - \int_{0}^{t} e^{-i\tau\abs{\xi}} \cF\Omega(u, \ol{Nu})(\tau,\xi)\,d\tau},
\end{align*}
where
\begin{equation}\label{bilinear-op-theta}
\begin{aligned}
  \Theta(f,g)(x)&=(2\pi)^{-\frac{3}{2}}\cF^{-1}\int  \frac{\cP_{XL+LX}(\xi-\eta,\eta)}{\theta(\xi-\eta,\eta)}\, \wh{f}(\xi-\eta)\,\wh{g}(\eta)\,d\eta\\
  &=(2\pi)^{-3}\iint e^{ix\cdot(\xi+\eta)}\,\frac{\cP_{XL+LX}(\xi,\eta)}{\theta(\xi,\eta)}\, \wh{f}(\xi)\,\wh{g}(\eta)\,d\xi d\eta.
\end{aligned}
\end{equation}
Finally, we obtain
\begin{equation}\label{int-normal-wave}
\begin{aligned}
  N(t)&=W(t)N_{0}  + W(t)\vD^{\gamma}\Theta(u_{0}, \ol{u_{0}}) 
    - \vD^{\gamma}\Theta(u(t), \ol{u(t)})\\
    &\qquad {} -i \int_{0}^{t} W(t-\tau)\vD^{\gamma} (u\,\ol{u})_{RL+LR+HH}(\tau)\,d\tau\\
    &\qquad {} -i \int_{0}^{t} W(t-\tau)\vD^{\gamma} \Theta\rb{Nu, \ol{u}}(\tau)\,d\tau\\
   &\qquad {} +i \int_{0}^{t} W(t-\tau) \vD^{\gamma} \Theta(u, \ol{Nu})(\tau)\,d\tau.
\end{aligned}
\end{equation}

\subsection{Main Results}

Before stating our main results, we summarize notation of function spaces which we use.
By using notation in section \ref{notation}, the norm of homogeneous Besov space is
\[
  \norm{f}_{\dot{B}^{s}_{q,2}}=\RB{\sum_{j\in\Z} 2^{2sj}\norm{P_{j}f}_{L^{q}}^{2}}^{\frac{1}{2}}.
\]
We only use the space which the third exponent is $2$, so we omit it below so that $\dot{B}^{s}_{q}=\dot{B}^{s}_{q,2}$.
The norm of the Sobolev space is
\[
  \norm{f}_{H^{s}}=\norm{\aD^{s} f }_{L^{2}}
  \simeq \RB{\sum_{j\in\Z} \ab{2^{j}}^{2s}\norm{P_{j}f}_{L^{2}}^{2}}^{\frac{1}{2}}.
\]
For the use of radial Strichartz estimates in section \ref{sec-radial-strichartz} we set
\[
  \frac{1}{q(\ve)}=\frac{1}{4}+\frac{\ve}{3}.
\]
We fix $0<\ve\ll 1$ so that
\[
  \frac{10}{3}< q(\ve) < 4 < q(-\ve),
\]
then $(2,q(\ve))$ is radial Strichartz admissible, and $(2,q(-\ve))$ is wave radial admissible.
(See Proposition \ref{radial-Strichartz} below.)
Since the space dimension is $3$, we frequently use the following embedding
\[
  H^{1}\subset \dot{B}^{\frac{1}{4}+\ve}_{q(\ve)} \subset L^{6},\qquad
  \dot{H}^{\frac{1}{2}}\subset \dot{B}^{-\frac{1}{4}-\ve}_{q(-\ve)}.
\]

Now we are in a position to state our main result.

\begin{theorem}\label{thm-1}
Let $\gamma\in [\frac{1}{2},1]$.
We  assume that $u_{0}$, $N_{0}$ are radial and satisfy $\norm{u_{0}}_{H^{1}}+\norm{N_{0}}_{H^{\frac{1-\gamma}{2}}}\le \rho$.
If $\rho>0$ is sufficiently small, then there exists a unique global solution
\begin{equation}\label{cond-conti}
  u\in C\rb{[0,\infty);H^{1}},\quad N\in C\rb{[0,\infty); H^{\frac{1-\gamma}{2}}}
\end{equation}
to  \eqref{int-normal-schrodinger}, \eqref{int-normal-wave} satisfying
\begin{equation}\label{cond-L2}
  \aD u\in L^{\infty}_{t}L^{2}_{x}\cap L^{2}_{t}\dot{B}^{\frac{1}{4}+\ve}_{q(\ve)},\quad
  \aD^{\frac{1-\gamma}{2}} N \in L^{\infty}_{t}L^{2}_{x}\cap L^{2}_{t}\dot{B}^{-\frac{1}{4}-\ve}_{q(-\ve)}.
\end{equation}

Moreover, there exists $u_{+}\in H^{1}$ and $N_{+}\in H^{\frac{1-\gamma}{2}}$ such that
\[
  \lim_{t\to\infty}\norm{u(t)-e^{-it\Delta}u_{+}}_{H^{1}}=0,\qquad
  \lim_{t\to\infty}\norm{N(t)-e^{it\vD}N_{+}}_{H^{\frac{1-\gamma}{2}}}=0.
\]
\end{theorem}

\medskip
\begin{remark}
The case $\gamma =1$ is the same one as \cite[Theorem 1.1]{guo-nakanishi_1}.
In their proof of the scattering, they just mentioned \eqref{appendix-boundary-limit} holds.
So, we give its proof in section \ref{sec-appendix}, because we need Lemma \ref{lem-boundary-4} additionally.
\end{remark}

\begin{remark}
The condition $\gamma \ge \frac{1}{2}$ essentially come from the estimate of the quadratic term in Lemma \ref{lem-quadratic-2} (1).
\end{remark}

\bigskip
The integral equations \eqref{int-normal-schrodinger}, \eqref{int-normal-wave}, which we show the existence of a solution, is not the usual one. 
Our second result states that solutions to the integral equations \eqref{int-normal-schrodinger}, \eqref{int-normal-wave} derived in Theorem \ref{thm-1} also satisfy the original system, which is not proved in \cite{guo-nakanishi_1} rigorously.

\begin{theorem}\label{thm-2}
Let $\gamma\in [\frac{1}{2},1]$.
We assume that $(u, N)$ is a radial solution to \eqref{int-normal-schrodinger}, \eqref{int-normal-wave} satisfying \eqref{cond-conti} and
\[
  \norm{\aD u}_{L^{\infty}_{t}L^{2}_{x}\cap L^{2}_{t}\dot{B}^{\frac{1}{4}+\ve}_{q(\ve)}}
  +\norm{\aD^{\frac{1-\gamma}{2}} N}_{L^{\infty}_{t}L^{2}_{x}\cap L^{2}_{t}\dot{B}^{-\frac{1}{4}-\ve}_{q(-\ve)}}\le \eta
\]
for sufficiently small $\eta>0$.
Then 
\begin{equation}\label{thm-2-c1}
  u\in C^{1}\rb{(0,\infty); H^{-1}},\quad N\in C^{1}\rb{(0,\infty); H^{-\frac{1+\gamma}{2}}}
\end{equation}
and satisfy \eqref{differential-equation}.
\end{theorem}

\section{Basic estimates}

In this section, we summarize the estimates which is required to the proof of theorems.

\subsection{Radial Strichartz estimates}\label{sec-radial-strichartz}

As in Guo-Nakanishi \cite{guo-nakanishi_1}, we employ the radial Strichartz estimates, which is proved in \cite{guo-wang_2014}.

\begin{proposition}\label{radial-Strichartz}
Let the space dimension $n=3$.
We assume $f(x)$ and $F(t,x)$ are radial in space variables.

\medskip
\noindent
(1) If $(p,q)$ and $(\wt{p},\wt{q})$ both satisfy the radial Schr\"{o}diger-admissible condition:
\[
  p\in [2,\infty],\quad \frac{2}{p}+\frac{5}{q}<\frac{5}{2} \quad\text{or}\quad
  (p,q)=(\infty,2)
\]
and $\wt{p}>2$, then
\begin{align*}
  &\norm{e^{-it\Delta}f}_{L^{p}_{t}\dot{B}^{\frac{2}{p}+\frac{3}{q}-\frac{3}{2}}_{q,2}}
    \lesssim \norm{f}_{L^{2}},\\
  &\Norm{\int_{0}^{t}e^{-i(t-\tau)\Delta}F(\tau)\,d\tau}_{L^{p}_{t}\dot{B}^{\frac{2}{p}+\frac{3}{q}-\frac{3}{2}}_{q,2}}
    \lesssim \norm{F}_{L^{\tilde{p}'}_{t}\dot{B}^{\frac{3}{2}-\frac{2}{\tilde{p}}-\frac{3}{\tilde{q}}}_{\tilde{q}',2}}.
\end{align*}

\medskip
\noindent
(2) If $(p,q)$ and $(\wt{p},\wt{q})$ both satisfy the radial wave-admissible condition:
\[
  p\in [2,\infty],\quad \frac{1}{p}+\frac{2}{q}<1 \quad\text{or}\quad
  (p,q)=(\infty,2)
\]
and $\wt{p}>2$, then
\begin{align*}
  &\norm{e^{it\vD} f}_{L^{p}_{t}\dot{B}^{\frac{1}{p}+\frac{3}{q}-\frac{3}{2}}_{q,2}}
    \lesssim \norm{f}_{L^{2}},\\
  &\Norm{\int_{0}^{t}e^{i(t-\tau)\vD}F(\tau)\,d\tau}_{L^{p}_{t}\dot{B}^{\frac{1}{p}+\frac{3}{q}-\frac{3}{2}}_{q,2}}
    \lesssim \norm{F}_{L^{\tilde{p}'}_{t}\dot{B}^{\frac{3}{2}-\frac{1}{\tilde{p}}-\frac{3}{\tilde{q}}}_{\tilde{q}',2}}.
\end{align*}
\end{proposition}

\begin{remark}
When $p=2$, the radial Schr\"{o}dinger-admissible condition implies $q>\frac{10}{3}$, while the radial wave-admissible condition implies $q>4$.
So, we apply this theorem by setting $\frac{1}{q(\ve)}=\frac{1}{4}+\frac{\ve}{3}$ with $\ve >0$ which satisfy
\[
  \frac{10}{3}<q(\ve)<4<q(-\ve).
\]
Then, for example,
\[
  \norm{S(t)f}_{L^{2}_{t}\dot{B}^{\frac{1}{4}+\ve}_{q(\ve),2}}
  \lesssim \norm{f}_{L^{2}},\qquad
  \norm{W(t)f}_{L^{2}_{t}\dot{B}^{-\frac{1}{4}-\ve}_{q(-\ve),2}}
  \lesssim \norm{f}_{L^{2}},
\]
hold, and $\frac{1}{2}=\frac{1}{q(\ve)}+\frac{1}{q(-\ve)}$ which is used to treat the quadratic nonlinearity.
\end{remark}

\subsection{Bilinear Fourier multiplier}\label{subsec-bilinear-FM}

For $m\in L^{\infty}(\R^{n}\times\R^{n})$, we denote the bilinear Fourier multiplier as
\[
  T_{m}(f,g)(x)=(2\pi)^{-n}\int_{\R^{n}\times\R^{n}}
    e^{ix\cdot(\xi+\eta)}m(\xi,\eta)\wh{f}(\xi)\,\wh{g}(\eta)\,d\xi d\eta.
\]
We employ the following proposition to estimate the operators $\Omega(f,g)$, $\Theta(f,g)$ appeared in \eqref{int-normal-schrodinger}, \eqref{int-normal-wave}, respectively.

\begin{proposition}[{\cite[Lemma 3.5]{guo-nakanishi_1}}]\label{bilinear-Fourier-multiplier}
Let $m\in C^{[n+1]}\Rb{\rb{\R^{n}\setminus\cb{0}}\times \rb{\R^{n}\setminus\cb{0}}}$ satisfy
\begin{equation}\label{multiplier-cond}
  \abs{\d_{\xi}^{\alpha}\d_{\eta}^{\beta}m(\xi,\eta)}
    \le C_{\alpha\beta} \abs{\xi}^{-\abs{\alpha}} \abs{\eta}^{-\abs{\beta}},
    \quad \abs{\alpha},\ \abs{\beta}\le n+1.
\end{equation}
If $p$, $q$, $r\in [1,\infty]$ satisfy $\frac{1}{r}=\frac{1}{p}+\frac{1}{q}$, then
\[
  \norm{T_{m}(P_{k} f, P_{l} g)}_{L^{r}}\le C\norm{f}_{L^{p}}\norm{g}_{L^{q}},
  \quad k,\ l\in\Z.
\]
\end{proposition}

\bigskip
By \eqref{bilinear-op-omega}, \eqref{bilinear-op-theta}, the multiplier of the operators $\vD\aD\Omega(f,g)$, $\vD\aD\Theta(f,g)$ are
\[
  m_{\Omega}(\xi,\eta)=\frac{\abs{\xi+\eta}\ab{\xi+\eta}}{\omega(\xi,\eta)}\,\cP_{XL}(\xi,\eta),\qquad
  m_{\Theta}(\xi,\eta)=\frac{\abs{\xi+\eta}\ab{\xi+\eta}}{\theta(\xi,\eta)}\,\cP_{XL+LX}(\xi,\eta),\
\]
respectively, where
\[
  \omega(\xi,\eta)=-\abs{\xi+\eta}^{2} + \abs{\xi} + \abs{\eta}^{2},\qquad
  \theta(\xi,\eta)=-\abs{\xi+\eta} + \abs{\xi}^{2} - \abs{\eta}^{2}.
\]
Below we first show $m_{\Omega}(\xi,\eta)$ satisfies \eqref{multiplier-cond}.
Since $\cP_{XL}=\sum_{\abs{j}\ge 3}\phi_{j}(\xi)\, \vp_{j-4}(\eta)$, it suffices to show the estimate holds for
\[
  m_{\Omega}^{j}(\xi,\eta)=\frac{\abs{\xi+\eta}\ab{\xi+\eta}}{\omega(\xi,\eta)}\,
    \phi_{j}(\xi)\, \vp_{j-4}(\eta).
\]
Note that for fixed $(\xi,\eta)$ there are at most three nonzero terms on the component of $\cP_{XL}$.

Here, we notice that $(\xi,\eta)\in \supp m^{j}_{\Omega}$ implies that 
\begin{equation}\label{HL-interaction}
  \frac{5}{8}\,2^{j}\le |\xi|\le \frac{8}{5}\,2^{j},\quad
  |\eta|\le \frac{8}{5}\,2^{j-4}<2^{j-3},\quad
  2^{j-1}<|\xi +\eta|<2^{j+1}.
\end{equation}
Then, if $j\ge 3$, we have
\[
  \abs{\omega(\xi,\eta)}\ge \abs{\xi+\eta}^{2} - \abs{\xi} - \abs{\eta}^{2}
    \ge 2^{2j-2}-\frac{8}{5}\,2^{j} -2^{2j-6}
    \ge \RB{\frac{1}{4}-\frac{1}{5}-\frac{1}{64}}\,2^{2j}.
\]
And if $j\le -3$, we have
\[
  \abs{\omega(\xi,\eta)}\ge \abs{\xi} -\abs{\xi+\eta}^{2} + \abs{\eta}^{2}
    \ge \frac{5}{8}\, 2^{j} - 2^{2j+2}
    \ge \RB{\frac{5}{8}-\frac{4}{8}}\,2^{j}.
\]
Therefore, we have
\begin{equation}\label{omega-est}
  \abs{m_{\Omega}^{j}(\xi,\eta)}\le C.
\end{equation}

The estimate of $\abs{\d_{\xi}^{\alpha}\d_{\eta}^{\beta}m(\xi,\eta)}$ can be done similarly.

\medskip
As for $m_{\Theta}(\xi,\eta)$, it suffices to show the estimate holds for
\[
  m_{\Theta_{XL}}^{j}(\xi,\eta)=\frac{\abs{\xi+\eta}\ab{\xi+\eta}}{\theta(\xi,\eta)}\,
    \phi_{j}(\xi)\, \vp_{j-4}(\eta).
\]
Since $(\xi,\eta)\in\supp  m_{\Theta_{XL}}^{j}$ implies \eqref{HL-interaction}, 
\[
  \abs{\theta(\xi,\eta)}\ge \abs{\xi}^{2}-\abs{\xi+\eta}-\abs{\eta}^{2}
  \ge \frac{25}{64}\, 2^{2j} - 2^{j+1} - 2^{2j-6}
  \ge \RB{\frac{25}{64}-\frac{1}{4}-\frac{1}{64}}2^{2j},
\]
provided that $j\ge 3$.
And if $j\le -3$, we have
\[
  \abs{\theta(\xi,\eta)}\ge \abs{\xi+\eta} -\abs{\xi}^{2} + \abs{\eta}^{2}
    \ge 2^{j-1}-\frac{64}{25}\, 2^{2j}
    \ge \RB{\frac{1}{2}-\frac{8}{25}}\,2^{j}.
\]
Therefore, we have
\begin{equation}\label{theta-est}
  \abs{m_{\Theta_{XL}}^{j}(\xi,\eta)}\le C.
\end{equation}
The estimate of  $\abs{\d_{\xi}^{\alpha}\d_{\eta}^{\beta}m(\xi,\eta)}$ can be done similarly.


\section{Estimates of the nonlinear terms}

We first state the estimate of the terms appeared in the paraproduct decomposition \eqref{paraproduct-decomp}, which we frequently use.

\begin{lemma}\label{paraproduct-est}
(1) For $\lambda\in\R$, $p\in [2,\infty)$,
\[
  \Norm{\vD^{\lambda} \rb{f g}_{HL}}_{L^{p}}
    \lesssim \RB{\sum_{j\in \Z} 2^{2\lambda j}\norm{ P_{j} f\cdot P_{\le j-4} g}_{L^{p}}^{2}}^{\frac{1}{2}}.
\]

\medskip
\noindent
(2) For $\lambda >0$, $p\in [2,\infty)$,
\[
  \Norm{\vD^{\lambda}\rb{f g}_{HH}}_{L^{p}}
    \lesssim  \sum_{\abs{l}\le 3}\sum_{j\in \Z} 2^{\lambda j} 
    \norm{P_{j}f\cdot P_{j+l}g}_{L^{p}}.
\]

\medskip
\noindent
(3) 
\[
  \norm{\rb{f g}_{HH}}_{L^{2}}
    \lesssim  \sum_{\abs{l}\le 3}\sum_{j\in \Z}
    \norm{P_{j}f\cdot P_{j+l}g}_{L^{2}}.  
\]
\end{lemma}

\begin{proof}
(1) For $j\in \Z$, considering \eqref{HL-interaction}, we have
\[
  \supp \cF\Sb{P_{j}f\cdot P_{\le j-4}g}\subset \cb{2^{j-1}< \abs{\xi}< 2^{j+1}}.
\]
Then, by the Littlewood-Paley theorem, we obtain
\begin{align*}
  \Norm{\vD^{\lambda}\rb{f g}_{HL}}_{L^{p}}^{2}
  & \lesssim \NM{\RB{\sum_{k\in\Z} \Abs{P_{k}\vD^{\lambda}\rb{f g}_{HL}}^{2}}^{\frac{1}{2}}}_{L^{p}}^{2}
  \le \sum_{k\in\Z} 2^{2\lambda k} \norm{P_{k} \rb{f g}_{HL}}_{L^{p}}^{2}\\
  &=\sum_{k\in\Z} 2^{2\lambda k} \NM{P_{k} \sum_{j\in \Z}P_{j}f\cdot P_{\le j-4}g}_{L^{p}}^{2}\\
  &  =\sum_{k\in\Z} 2^{2\lambda k} \NM{P_{k} \sum_{j=k-1}^{k+1} P_{j}f\cdot P_{\le j-4}g}_{L^{p}}^{2}\\
  &\lesssim \sum_{j\in \Z}\sum_{k=j-1}^{j+1} 2^{2\lambda k}
    \norm{ P_{j} f\cdot P_{\le j-4} g}_{L^{p}}^{2}.
\end{align*}
Thus, we have the desired estimate.

\medskip
\noindent
(2) For $j\in \Z$, $\abs{l}\le 3$, considering \eqref{phi-j-support-cond},
\[
  \supp \cF\Sb{P_{j}f\cdot P_{j+l}g}\subset \cb{\abs{\xi}< 2^{j+4}}.
\]
Then, by the Littlewood-Paley theorem, we obtain
\begin{align*}
  \Norm{\vD^{\lambda} \rb{f g}_{HH}}_{L^{p}}^{2}
  & \lesssim \sum_{k\in\Z} 2^{2\lambda k}\norm{P_{k} \rb{f g}_{HH}}_{L^{p}}^{2}\\
  & =\sum_{k\in\Z} 2^{2\lambda k} \NM{P_{k} \sum_{\abs{l}\le 3}\sum_{j\in Z}P_{j}f\cdot P_{j+l}g}_{L^{p}}^{2}\\
  & =\sum_{k\in\Z} 2^{2\lambda k} \NM{P_{k} \sum_{\abs{l}\le 3}\sum_{j\ge k-4}P_{j}f\cdot P_{j+l}g}_{L^{p}}^{2}\\
  &\lesssim \sum_{\abs{l}\le 3} \sum_{k\in Z}
    \RB{\sum_{j\ge k-4} 2^{\lambda k} \Norm{P_{j}f\cdot P_{j+l}g}_{L^{p}}}^{2}.
\end{align*}
Thus, since $\lambda >0$,
\begin{align*}
  \Norm{\vD^{\lambda} \rb{f g}_{HH}}_{L^{p}}
  &\le \sum_{\abs{l}\le 3}\sum_{j\in Z}\RB{ \sum_{k\le j+4} 2^{2\lambda k}\Norm{P_{j}f\cdot P_{j+l}g}_{L^{p}}^{2}}^{\frac{1}{2}}\\
  &\simeq \sum_{\abs{l}\le 3}\sum_{j\in Z} 2^{\lambda j}
    \norm{P_{j}f\cdot P_{j+l}g}_{L^{p}}.  
\end{align*}

\medskip
\noindent
(3) Following the proof of (2), we have
\begin{align*}
  \norm{\rb{f g}_{HH}}_{L^{2}}
  &\le \sum_{\abs{l}\le 3}\sum_{j\in Z}\RB{ \sum_{k\le j+4} 
    \Norm{P_{k}\rb{P_{j}f\cdot P_{j+l}g}}_{L^{2}}^{2}}^{\frac{1}{2}}\\
  &\le \sum_{\abs{l}\le 3}\sum_{j\in Z} 
    \Norm{\Rb{\sum_{k\in \Z} \abs{P_{k}\rb{P_{j}f\cdot P_{j+l}g}}^{2}}^{\frac{1}{2}}}_{L^{2}}\\
  &\lesssim \sum_{\abs{l}\le 3}\sum_{j\in Z}
    \norm{P_{j}f\cdot P_{j+l}g}_{L^{2}}.  \qedhere
\end{align*}
\end{proof}

\subsection{Quadratic terms}

The next lemma is the same as the one in Guo-Nakanishi \cite[Lemma 3.2]{guo-nakanishi_1}, but we give its proof here for reader's convenience.

\begin{lemma}\label{lem-quadratic-1}
Let $\frac{1}{q(\ve)}=\frac{1}{4}+\frac{\ve}{3}$.

\noindent
(1) For $\ve >0$,
\begin{align*}
  & \norm{(Nu)_{LH}}_{L^{1}_{t}H^{1}_{x}}
    \lesssim \norm{N}_{L^{2}_{t}\dot{B}^{-\frac{1}{4}-\ve}_{q(-\ve)}} 
    \norm{\aD u}_{L^{2}_{t}\dot{B}^{\frac{1}{4}+\ve}_{q(\ve)}},\\
  & \norm{(Nu)_{HH}}_{L^{1}_{t}H^{1}_{x}}
    \lesssim \norm{N}_{L^{2}_{t}\dot{B}^{-\frac{1}{4}-\ve}_{q(-\ve)}} 
    \norm{\aD u}_{L^{2}_{t}\dot{B}^{\frac{1}{4}+\ve}_{q(\ve)}}.
\end{align*}

\noindent
(2) For $\ve>0$, $\theta\in [0,1]$, $\frac{1}{a}=\frac{1}{2}-\frac{\theta}{2}$, $\frac{1}{b}=\frac{1}{q(\ve)}+\frac{\theta}{3}$,
\[
  \norm{\aD (Nu)_{RL}}_{L^{a'}_{t}\dot{B}^{\frac{3}{2}-\frac{2}{a}-\frac{3}{b}}_{b'}}
    \lesssim \norm{N}_{L^{2}_{t}\dot{B}^{-\frac{1}{4}-\ve}_{q(-\ve)}} 
    \norm{u}_{L^{\infty}_{t}L^{2}\cap L^{2}_{t}\dot{B}^{\frac{1}{4}+\ve}_{q(\ve)}}.
\] 
\end{lemma}

\begin{remark}
In (2), $(a,b)$ is radial Schr\"{o}dinger admissible if $\theta\in (0,\frac{3}{8})$ and $\ve$ is sufficiently small.
In application to the proof of theorems, we take $\theta=\frac{3}{8}-\frac{5}{2}\,\ve$ so that $(a,b)$ defined above is radial Schr\"{o}dinger-admissible without any restriction on $\ve>0$.
\end{remark}

\begin{proof}
(1) Applying Lemma \ref{paraproduct-est}, we have
\begin{align*}
   \norm{\aD (Nu)_{LH}}_{L^{2}}
   &\lesssim \RB{\sum_{j\in \Z} \ab{2^{j}}^{2}\norm{ P_{\le j-4} N\cdot P_{j} u }_{L^{2}}^{2}}^{\frac{1}{2}}\\
    &\lesssim \RB{\sum_{j\in \Z} \ab{2^{j}}^{2} 
      \norm{ P_{\le j-4} N}_{L^{q(-\ve)}}^{2}
      \norm{P_{j} u }_{L^{q(\ve)}}^{2}}^{\frac{1}{2}}.
\end{align*}
Since
\begin{equation}\label{low-freq-est}
\begin{aligned}
  \norm{ P_{\le j-4} N}_{L^{q(-\ve)}} 
  &\le \sum_{k\le j-4} 2^{(\frac{1}{4}+\ve)k}\,2^{(-\frac{1}{4}-\ve)k}\norm{P_{k}N}_{L^{q(-\ve)}}\\
  &\le \RB{\sum_{k\le j-4} 2^{2(\frac{1}{4}+\ve)k}}^{\frac{1}{2}}
    \RB{\sum_{k\le j-4} 2^{2(-\frac{1}{4}-\ve)k}\norm{P_{k}N}_{L^{q(-\ve)}}^{2}}^{\frac{1}{2}}\\
  &=C 2^{(\frac{1}{4}+\ve)j}  \norm{N}_{\dot{B}^{-\frac{1}{4}-\ve}_{q(-\ve)}},
\end{aligned}
\end{equation}
we obtain the desired result.

\medskip
\noindent
Similarly, applying Lemma \ref{paraproduct-est}, we have
\begin{align*}
   \norm{\aD (Nu)_{HH}}_{L^{2}}
   &\lesssim  \sum_{\abs{l}\le 3}\sum_{j\in \Z} \ab{2^{j}} 
    \norm{P_{j}N\cdot P_{j+l}u}_{L^{2}}\\
   &\lesssim  \sum_{\abs{l}\le 3}\sum_{j\in \Z} \ab{2^{j}}
    \norm{P_{j}N}_{L^{q(-\ve)}} \norm{P_{j+l}u}_{L^{q(\ve)}}\\
  &\le \sum_{\abs{l}\le 3}\RB{\sum_{j\in \Z} 2^{2(-\frac{1}{4}-\ve)j}  \norm{P_{j}N}_{L^{q(-\ve)}}^{2}}^{\frac{1}{2}}
  \RB{\sum_{j\in \Z} 2^{2(\frac{1}{4}+\ve)j}  \norm{P_{j+l}\aD u}_{L^{q(\ve)}}^{2}}^{\frac{1}{2}}. 
\end{align*}
Thus, we obtain the desired estimate.

\medskip
\noindent
(2) Since $\supp \cF{(Nu)_{RL}}\subset \cb{ 2^{-3}\le \abs{\xi}\le 2^{3}}$,
\begin{align*}
  \norm{\aD (Nu)_{RL}}_{L^{p'}_{t}\dot{B}^{\frac{3}{2}-\frac{2}{p}-\frac{3}{q}}_{q'}}
  \lesssim \norm{(Nu)_{RL}}_{L^{p'}_{t}L^{q'}_{x}}
  \le  \sum_{\abs{j}\le 2}\norm{P_{j} N}_{L^{2}_{t}L^{q(-\ve)}_{x}} \norm{P_{\le j-4}u}_{L^{\frac{2}{\theta}}_{t}L^{\frac{6}{3-2\theta}}},
\end{align*}
where we used the relation $\frac{1}{p'}=\frac{1}{2}+\frac{\theta}{2}$, $\frac{1}{q'}=\frac{1}{q(-\ve)}+\frac{3-2\theta}{6}$.
And since $\frac{\theta}{2}=\frac{1-\theta}{\infty}+\frac{\theta}{2}$, $\frac{3-2\theta}{6}=\frac{1-\theta}{2}+\frac{\theta}{6}$, we have
\[
  \norm{P_{\le j-4}u}_{L^{\frac{2}{\theta}}_{t}L^{\frac{6}{3-2\theta}}}
  \lesssim \norm{u}_{L^{\frac{2}{\theta}}_{t}L^{\frac{6}{3-2\theta}}}
  \le \norm{u}_{L^{\infty}_{t}L^{2}_{x}}^{1-\theta}\norm{u}_{L^{2}_{t}L^{6}_{x}}^{\theta}
  \lesssim \norm{u}_{L^{\infty}_{t}L^{2}_{x}}
    +\norm{u}_{L^{2}_{t}\dot{B}^{\frac{1}{4}+\ve}_{q(\ve)}}.
\]
Therefore, we obtain the desired estimate.
\end{proof}

\smallskip
\begin{lemma}\label{lem-quadratic-2}
We set $\frac{1}{q(\ve)}=\frac{1}{4}+\frac{\ve}{3}$.

\noindent
(1) Let $\gamma\in [\frac{1}{2},1]$. 
For $\ve >0$,
\begin{align*}
  & \norm{\aD^{\frac{1-\gamma}{2}}\vD^{\gamma} (u v)_{HH}}_{L^{1}_{t}L^{2}_{x}}
    \lesssim \norm{u}_{L^{2}_{t}\dot{B}^{\frac{1}{4}+\ve}_{q(\ve)}} 
    \norm{\aD^{\frac{\gamma}{2}} v}_{L^{2}_{t}\dot{B}^{\frac{1}{4}+\ve}_{q(\ve)}}.
\end{align*}

\noindent
(2) Let $\gamma\in (0,1]$.
For $\ve>0$, $\theta\in [0,1]$, $\frac{1}{a}=\frac{1}{2}-\frac{\theta}{2}$, $\frac{1}{b}=\frac{1}{q(-\ve)}+\frac{\theta}{3}$,
\[
  \norm{\aD^{\frac{1-\gamma}{2}}\vD^{\gamma} (uv)_{RL}}_{L^{a'}_{t}\dot{B}^{\frac{3}{2}-\frac{1}{a}-\frac{3}{b}}_{b'}}
    \lesssim 
    \norm{ u}_{L^{\infty}_{t}L^{2}\cap L^{2}_{t}\dot{B}^{\frac{1}{4}+\ve}_{q(\ve)}}^{2}.
\] 
\end{lemma}

\begin{remark}
In (2), $(a,b)$ is radial wave admissible if $\theta\in (0,1]$ and $\ve$ is sufficiently small.
In application to the proof of theorems, we take $\theta=4\ve$ so that $(a,b)$ defined above is radial wave-admissible without any restriction on $\ve>0$.
\end{remark}

\begin{proof}
(1) Applying Lemma \ref{paraproduct-est}, we have
\begin{align*}
   \norm{\aD^{\frac{1-\gamma}{2}}\vD^{\gamma} (uv)_{HH}}_{L^{2}}
   &\lesssim  \sum_{\abs{l}\le 3}\sum_{j\in \Z} 
     \ab{2^{j}}^{\frac{1-\gamma}{2}}\,2^{\gamma j} 
    \norm{P_{j}u\cdot P_{j+l}v}_{L^{2}}\\
   &\lesssim  \sum_{\abs{l}\le 3}\sum_{j\in \Z} 
     \ab{2^{j}}^{\frac{1-\gamma}{2}}\,2^{(\gamma-\frac{1}{2}) j}\,
    2^{(\frac{1}{4}-\ve)j} \norm{P_{j}u}_{L^{q(-\ve)}} 
    2^{(\frac{1}{4}+\ve)j}\norm{P_{j+l}v}_{L^{q(\ve)}}\\
  &\lesssim \sum_{\abs{l}\le 3}\RB{\sum_{j\in \Z} 2^{2(\frac{1}{4}+\ve)j}  \norm{P_{j}u}_{L^{q(\ve)}}^{2}}^{\frac{1}{2}}
  \RB{\sum_{j\in \Z} 2^{2(\frac{1}{4}+\ve)j}  \norm{P_{j+l}\aD^{\frac{\gamma}{2}} u}_{L^{q(\ve)}}^{2}}^{\frac{1}{2}},
\end{align*}
where in the last inequality we used $2^{(\gamma-\frac{1}{2})j}\le \ab{2^{j}}^{\gamma -\frac{1}{2}}$, which holds when $\gamma\in [\frac{1}{2},1]$.
Thus, we obtain the desired estimate.

\medskip
\noindent
(2) Since $\supp \cF{(uv)_{RL}}\subset \cb{ 2^{-3}\le \abs{\xi}\le 2^{3}}$,
\begin{align*}
  \norm{\aD^{\frac{1-\gamma}{2}}\vD^{\gamma}  (uv)_{RL}}_{L^{p'}_{t}\dot{B}^{\frac{3}{2}-\frac{1}{p}-\frac{3}{q}}_{q'}}
  \lesssim \norm{(uv)_{RL}}_{L^{p'}_{t}L^{q'}_{x}}
  \le  \sum_{\abs{j}\le 2}\norm{P_{j} u}_{L^{2}_{t}L^{q(\ve)}_{x}} \norm{P_{\le j-4}v}_{L^{\frac{2}{\theta}}_{t}L^{\frac{6}{3-2\theta}}},
\end{align*}
where we used the relation $\frac{1}{p'}=\frac{1}{2}+\frac{\theta}{2}$, $\frac{1}{q'}=\frac{1}{q(\ve)}+\frac{3-2\theta}{6}$.
And since $\frac{\theta}{2}=\frac{1-\theta}{\infty}+\frac{\theta}{2}$, $\frac{3-2\theta}{6}=\frac{1-\theta}{2}+\frac{\theta}{6}$, we have
\[
  \norm{P_{\le j-4}v}_{L^{\frac{2}{\theta}}_{t}L^{\frac{6}{3-2\theta}}}
  \lesssim \norm{v}_{L^{\frac{2}{\theta}}_{t}L^{\frac{6}{3-2\theta}}}
  \le \norm{v}_{L^{\infty}_{t}L^{2}_{x}}^{1-\theta}\norm{v}_{L^{2}_{t}L^{6}_{x}}^{\theta}
  \lesssim \norm{v}_{L^{\infty}_{t}L^{2}_{x}}
    +\norm{v}_{L^{2}_{t}\dot{B}^{\frac{1}{4}+\ve}_{q(\ve)}}.
\]
Therefore, we obtain the desired estimate.
\end{proof}

\subsection{Boundary terms}

\smallskip
\begin{lemma}\label{lem-boundary-1}
(1) For $s>\frac{1}{2}$, $\sigma>-\frac{1}{2}$,
\begin{align*}
  &\norm{\Omega(N,u)}_{H^{1}} \lesssim \norm{N}_{L^{2}}
  \norm{u}_{H^{s}},\\
  &\norm{\Omega(N,u)}_{H^{-1}} \lesssim \norm{N}_{L^{\sigma}}
  \norm{u}_{H^{-1}},\qquad
  \norm{\Omega(N,u)}_{H^{-1}} \lesssim \norm{N}_{H^{-1}}
  \norm{u}_{H^{\sigma}}.
\end{align*}

\noindent
(2) For $\gamma\in (-\frac{1}{2},1]$, $s >\frac{\gamma}{4}$, $\sigma>-\frac{1}{2}$,
\begin{align*}
  &\norm{\aD^{\frac{1-\gamma}{2}}\vD^{\gamma}\Theta(u,v)}_{L^{2}} 
  \lesssim \norm{u}_{H^{s}} \norm{v}_{H^{s}},\\
  &\norm{\vD^{\gamma}\Theta(u,v)}_{H^{-\frac{1+\gamma}{2}}}
  \lesssim \norm{u}_{H^{-\frac{1+\gamma}{2}}} \norm{v}_{H^{\sigma}}
    + \norm{u}_{H^{\sigma}}\norm{v}_{H^{-\frac{1+\gamma}{2}}}.
\end{align*}
\end{lemma}

\begin{proof}
(1) We first recall that
\[
  \cF\Sb{\aD\Omega(N,u)}(\xi)
  = (2\pi)^{-\frac{3}{2}}\int \frac{\ab{\xi}}{\omega(\xi-\eta,\eta)}\,\cP_{XL}(\xi-\eta,\eta)\wh{N}(\xi-\eta)\wh{u}(\eta)\,d\eta.
\]
The argument to show \eqref{omega-est} implies
\begin{equation}\label{pf-lem-boundary-1_1}
  \ABS{\frac{\ab{\xi}}{\omega(\xi-\eta,\eta)}}\,
  \cP_{XL}(\xi-\eta,\eta)
  \lesssim  \frac{\ab{\xi-\eta}}{\ab{\xi-\eta}\abs{\xi-\eta}}\,\cP_{XL}(\xi-\eta,\eta)
  \lesssim \abs{\eta}^{-1}.
\end{equation}
Note that $(\xi-\eta,\eta)\in \supp\cP_{XL}$ implies $\abs{\eta}\ll \abs{\xi-\eta}\simeq \abs{\xi}$.
Thus,
\begin{align*}
  \norm{\aD\Omega(N,u)}_{L^{2}}
  &\lesssim\NM{\int \abs{\eta}^{-1}\abs{\wh{N}(\xi-\eta)}
   \abs{\wh{u}(\eta)}\,d\eta}_{L^{2}}\\
  &\le \Norm{\wh{N}}_{L^{2}}
  \Norm{\abs{\eta}^{-1}\abs{\wh{u}(\eta)}}_{L^{1}}\\
  &\le \Norm{\abs{\eta}^{-1}\,\ab{\eta}^{-s}}_{L^{2}}
    \norm{N}_{L^{2}}\Norm{\ab{\eta}^{s}\wh{u}(\eta)}_{L^{2}}.
\end{align*}
Since $s >\frac{1}{2}$ implies $2(-1-s)<-3$, we obtain the desired estimate.

\smallskip
For the second one, we estimate 
\begin{align*}
  \ABS{\frac{\cP_{XL}(\xi-\eta,\eta)}{\ab{\xi}\,\omega(\xi-\eta,\eta)}}
  &\lesssim  \frac{\cP_{XL}(\xi-\eta,\eta)}{\ab{\xi}\ab{\xi-\eta}\abs{\xi-\eta}}  \lesssim \ab{\eta}^{-1} \ab{\xi-\eta}^{-1}\abs{\xi-\eta}^{-1},
\end{align*}
since $(\xi-\eta,\eta)\in \supp\cP_{XL}$ implies $\abs{\eta}\ll \abs{\xi-\eta}\simeq \abs{\xi}$.
Thus,
\begin{align*}
  \norm{\aD\Omega(N,u)}_{L^{2}}
  &\lesssim \NM{\int \ab{\eta}^{-1}\ab{\xi-\eta}^{-1}\abs{\xi-\eta}^{-1}\abs{\wh{N}(\xi-\eta)}
   \abs{\wh{u}(\eta)}\,d\eta}_{L^{2}}\\
  &\le \Norm{\ab{\eta}^{-1}\abs{\eta}^{-1}\wh{N}(\eta)}_{L^{1}}
  \Norm{\ab{\eta}^{-1}\wh{u}(\eta)}_{L^{2}}\\
  &\le \Norm{\abs{\eta}^{-1}\,\ab{\eta}^{-1-\sigma}}_{L^{2}}
    \norm{\ab{\eta}^{\sigma}\wh{N}}_{L^{2}}\norm{u}_{H^{s}}.
\end{align*}
Since $\sigma >-\frac{1}{2}$ implies $2(-2-\sigma)<-3$, we obtain the desired estimate.

The third one is proved similarly.

\medskip
\noindent
(2) We recall that
\[
  \cF\Sb{\aD^{\frac{1-\gamma }{2}}\vD^{\gamma}\Theta(u,v)}(\xi)
  = (2\pi)^{-\frac{3}{2}}\int \frac{\ab{\xi}^{\frac{1-\gamma}{2}}\abs{\xi}^{\gamma}}{\theta(\xi-\eta,\eta)}\,\cP_{XL+LX}(\xi-\eta,\eta)\wh{u}(\xi-\eta)\wh{v}(\eta)\,d\eta.
\]
The argument to show \eqref{theta-est} implies
\begin{align*}
  \ABS{\frac{\ab{\xi}^{\frac{1-\gamma}{2}}\abs{\xi}^{\gamma}}{\theta(\xi-\eta,\eta)}}\,
  \cP_{XL}(\xi-\eta,\eta)
  &\lesssim  \frac{\ab{\xi-\eta}^{\frac{1-\gamma}{2}}\abs{\xi-\eta}^{\gamma}}{\ab{\xi-\eta}\abs{\xi-\eta}}\,\cP_{XL}(\xi-\eta,\eta)\\
  &\lesssim \ab{\xi-\eta}^{-\frac{1+\gamma}{4}}\abs{\xi-\eta}^{-\frac{1-\gamma}{2}}
    \ab{\eta}^{-\frac{1+\gamma}{4}}\abs{\eta}^{-\frac{1-\gamma}{2}}.
\end{align*}
Note that $(\xi-\eta,\eta)\in \supp\cP_{XL}$ implies $\abs{\eta}\ll \abs{\xi-\eta}\simeq \abs{\xi}$.
Thus,
\begin{align*}
  \norm{\aD^{\frac{1-\gamma }{2}}\vD^{\gamma}\Theta(u,v)}_{L^{2}}
  &\lesssim \NM{\int \ab{\xi-\eta}^{-\frac{1+\gamma}{4}}\abs{\xi-\eta}^{-\frac{1-\gamma}{2}}\abs{\wh{u}(\xi-\eta)}
   \ab{\eta}^{-\frac{1+\gamma}{4}}\abs{\eta}^{-\frac{1-\gamma}{2}}\abs{\wh{v}(\eta)}\,d\eta}_{L^{2}}\\
  &\le \Norm{\ab{\xi}^{-\frac{1+\gamma}{4}}\abs{\xi}^{-\frac{1-\gamma}{2}}\abs{\wh{u}(\xi)}}_{L^{\frac{4}{3}}}
  \Norm{\ab{\eta}^{-\frac{1+\gamma}{4}}\abs{\eta}^{-\frac{1-\gamma}{2}}\abs{\wh{v}(\eta)}}_{L^{\frac{4}{3}}}\\
  &\le \Norm{\ab{\xi}^{-\frac{1+\gamma}{4}-s}\,\abs{\xi}^{-\frac{1-\gamma}{2}}}_{L^{4}}^{2}
    \norm{u}_{H^{s}}\norm{v}_{H^{s}}.
\end{align*}
Since $\gamma >-\frac{1}{2}$ implies $-2(1-\gamma)>-3$, and $s >\frac{\gamma}{4}$ implies  $-(1+\gamma +4s +2(1-\gamma))<-3$, we obtain the desired estimate.

\smallskip
For the second one, we estimate 
\begin{align*}
  \ABS{\frac{\ab{\xi}^{-\frac{1+\gamma}{2}}\abs{\xi}^{\gamma}}{\theta(\xi-\eta,\eta)}}\,
  \cP_{XL}(\xi-\eta,\eta)
  &\lesssim  \frac{\ab{\eta}^{-\frac{1+\gamma}{2}}\abs{\xi-\eta}^{\gamma}}{\ab{\xi-\eta}\abs{\xi-\eta}}\,\cP_{XL}(\xi-\eta,\eta)\\
  &\lesssim \ab{\xi-\eta}^{-1}\abs{\xi-\eta}^{-1+\gamma}
    \ab{\eta}^{-\frac{1+\gamma}{2}}.
\end{align*}
since $(\xi-\eta,\eta)\in \supp\cP_{XL}$ implies $\abs{\eta}\ll \abs{\xi-\eta}\simeq \abs{\xi}$.
Thus,
\begin{align*}
  \norm{\vD^{\gamma}\Theta_{XL}(u,v)}_{H^{-\frac{1+\gamma}{2}}}
  &\lesssim\NM{\int \ab{\xi-\eta}^{-1}\abs{\xi-\eta}^{-1+\gamma}\abs{\wh{u}(\xi-\eta)}
   \ab{\eta}^{-\frac{1+\gamma}{2}}\abs{\wh{v}(\eta)}\,d\eta}_{L^{2}}\\
  &\le \Norm{\ab{\eta}^{-1}\abs{\eta}^{-1+\gamma}\wh{u}(\eta)}_{L^{1}}
  \Norm{\ab{\eta}^{-\frac{1+\gamma}{2}}\wh{v}(\eta)}_{L^{2}}\\
  &\le \Norm{\abs{\eta}^{-1}\,\ab{\eta}^{-1-\sigma}}_{L^{2}}
    \norm{\ab{\eta}^{\sigma}\wh{u}}_{L^{2}}\norm{v}_{H^{-\frac{1+\gamma}{2}}}.
\end{align*}
Since $\sigma >-\frac{1}{2}$ implies $2(-2-\sigma)<-3$, we obtain the desired estimate.
\end{proof}

\smallskip
\begin{lemma}\label{lem-boundary-2}
Let $\gamma\in [0,1]$.
We set $\frac{1}{q(\ve)}=\frac{1}{4}+\frac{\ve}{3}$.
For $\ve>0$,
\begin{align*}
  \norm{\aD\Omega(N,u)}_{L^{2}_{t}\dot{B}^{\frac{1}{4}+\ve}_{q(\ve)}} 
  &\lesssim  \norm{N}_{L^{\infty}_{t}L^{2}_{x}}
  \norm{\aD u}_{L^{2}_{t}L^{6}_{x}},\\
  \Norm{\aD^{\frac{1-\gamma}{2}}\vD^{\gamma}\Theta(u,v)}_{L^{2}_{t}\dot{B}^{-\frac{1}{4}-\ve}_{q(-\ve)}} 
  &\lesssim \norm{u}_{L^{\infty}_{t}L^{2}_{x}}\norm{v}_{L^{2}_{t}L^{6}_{x}}
  + \norm{u}_{L^{2}_{t}L^{6}_{x}} \norm{v}_{L^{\infty}_{t}L^{2}_{x}}.
\end{align*}
\end{lemma}

\begin{proof}
Since the first estimate is the same one as in \cite[Lemma 3.6]{guo-nakanishi_1}, we show the second estimate here.
For the second estimate, it suffices to prove
\[
  \Norm{\aD^{\frac{1-\gamma}{2}}\vD^{\gamma}\Theta_{XL}(u,v)}_{\dot{B}^{-\frac{1}{4}-\ve}_{q(-\ve)}} 
  \lesssim \norm{u}_{L^{2}} \norm{v}_{L^{6}}.
\]
We first apply the Sobolev embedding, and then apply Lemma \ref{HL-interaction} to estimate the HL-type interaction to obtain
\begin{align*}
  \Norm{\aD^{\frac{1-\gamma}{2}}\vD^{\gamma}\Theta_{XL}(u,v)}_{\dot{B}^{-\frac{1}{4}-\ve}_{q(-\ve)}}
  &\lesssim \Norm{\aD^{\frac{1-\gamma}{2}}\vD^{\gamma+\frac{1}{2}}\Theta_{XL}(u,v)}_{L^{2}}\\
  &= \Norm{\sum_{j\in\Z} \aD^{\frac{1-\gamma}{2}}\vD^{\gamma+\frac{1}{2}}\Theta_{XL}(P_{j}u, P_{\le j-3} v)}_{L^{2}}\\
  &\lesssim \RB{\sum_{j\in \Z} \ab{2^{j}}^{-1-\gamma}\,2^{(2\gamma -1) j}
  \norm{\aD\vD \Theta_{XL}(P_{j} u, P_{\le j-3} v)}_{L^{2}}^{2}}^{\frac{1}{2}}.
\end{align*}
Then, applying Proposition \ref{bilinear-Fourier-multiplier}, we obtain
\begin{align*}
  \norm{\aD \vD\Theta_{XL}(P_{j}u, P_{\le j-3} v)}_{L^{2}}
  &\le  \sum_{k\le j-3} \norm{\aD \vD\Theta_{XL}(P_{j}u, P_{k} v)}_{L^{2}}\\
  &\lesssim \sum_{k\le j-3} \norm{P_{j}u}_{L^{2}}\norm{P_{k}v}_{L^{\infty}}\\
  &\lesssim \sum_{k\le j-3} 2^{\frac{k}{2}}\norm{P_{j}u}_{L^{2}} \norm{v}_{L^{6}}
  \lesssim 2^{\frac{j}{2}}\norm{P_{j}u}_{L^{2}} \norm{v}_{L^{6}}.
\end{align*}
Therefore, we obtain
\begin{align*}
  \Norm{\aD^{\frac{1-\gamma}{2}}\vD^{\gamma}\Theta_{XL}(u,v)}_{\dot{B}^{-\frac{1}{4}-\ve}_{q(-\ve)}}
  &\lesssim \norm{v}_{L^{6}}\RB{\sum_{j\in \Z} \ab{2^{j}}^{-1-\gamma}\,2^{2\gamma j}
  \norm{P_{j}u}_{L^{2}}^{2}}^{\frac{1}{2}}\\
  &\lesssim \Norm{\aD^{-\frac{1-\gamma}{2}}u}_{L^{2}} \norm{v}_{L^{6}}
  \lesssim \norm{u}_{L^{2}} \norm{v}_{L^{6}}.
  \qedhere
\end{align*}
\end{proof}

The next lemma is used to estimate one of the cubic terms, and also to show the scattering.

\begin{lemma}\label{lem-boundary-3}
Let $\gamma\in [\frac{1}{2},1]$.
We set $\frac{1}{q(\ve)}=\frac{1}{4}+\frac{\ve}{3}$.
For $\ve>0$,
\begin{align*}
  &\norm{\aD\Omega(N, u)}_{L^{2}} 
  \lesssim  \norm{N}_{\dot{B}^{-\frac{1}{4}-\ve}_{q(-\ve)}} \norm{u}_{L^{2}},\\
  &\Norm{\aD^{\frac{1-\gamma}{2}}\vD^{\gamma}\Theta(u,v)}_{L^{2}}
    \lesssim \norm{u}_{L^{2}}\norm{\aD v}_{L^{6}}+ \norm{\aD u}_{L^{6}}\norm{v}_{L^{2}}.
\end{align*}
\end{lemma}

\begin{proof}
We first apply Lemma \ref{HL-interaction} to estimate the HL-type interaction to obtain
\begin{align*}
  \norm{\aD\Omega(N, u)}_{L^{2}_{x}} 
  &= \Norm{\sum_{j\in\Z} \aD\Omega(P_{j}N, P_{\le j-3}u)}_{L^{2}}\\
  &\lesssim \RB{\sum_{j\in \Z} 2^{-2j}
  \norm{\aD\vD \Omega(P_{j} N, P_{\le j-3}u)}_{L^{2}}^{2}}^{\frac{1}{2}}.
\end{align*}
Now we apply Proposition \ref{bilinear-Fourier-multiplier} to obtain
\begin{align*}
  \norm{\aD\vD \Omega(P_{j} N, P_{\le j-3}u)}_{L^{2}}
  &\le  \sum_{k\le j-3} \norm{\aD\vD \Omega(P_{j} N, P_{k} u)}_{L^{2}}\\
  &\lesssim \sum_{k\le j-3} \norm{P_{j}N}_{L^{q(-\ve)}}\norm{P_{k}u}_{L^{q(\ve)}}\\
  &\lesssim \sum_{k\le j-3} \norm{P_{j}N}_{L^{q(-\ve)}}\, 
    2^{(\frac{3}{4}-\ve)k}\norm{P_{k}u}_{L^{2}}\\
  &\lesssim 2^{(\frac{3}{4}-\ve)j} \norm{P_{j}N}_{L^{q(-\ve)}} \norm{u}_{L^{2}}.
\end{align*}
Thus, we obtain
\begin{align*}
  \norm{\aD\Omega(N, u)}_{L^{2}_{x}} 
  &\lesssim  \RB{\sum_{j\in\Z} 2^{2(-\frac{1}{4}-\ve)j} \norm{P_{j}N}_{L^{(q(-\ve)}}^{2}}^{\frac{1}{2}} \norm{u}_{L^{2}}.
\end{align*}

\smallskip
As for the second estimate, we apply Lemma \ref{HL-interaction} to estimate the HL-type interaction to obtain
\begin{align*}
  \norm{\aD^{\frac{1-\gamma}{2}}\vD^{\gamma}\Theta_{XL}(u,v)}_{L^{2}_{x}} 
  &= \Norm{\sum_{j\in\Z} \aD^{\frac{1-\gamma}{2}}\vD^{\gamma}\Theta_{XL}(P_{j}u, P_{\le j-3}v)}_{L^{2}}\\
  &\lesssim \RB{\sum_{j\in \Z} \ab{2^{j}}^{-1-\gamma}\, 2^{2(\gamma -1)j}
  \norm{\aD\vD \Theta_{XL}(P_{j}u, P_{\le j-3}v)}_{L^{2}}^{2}}^{\frac{1}{2}}.
\end{align*}
Now we apply Proposition \ref{bilinear-Fourier-multiplier} to obtain
\begin{align*}
  \norm{\aD\vD \Theta_{XL}(P_{j}u, P_{\le j-3}v)}_{L^{2}}
  &\le  \sum_{k\le j-3} \norm{\aD\vD \Theta_{XL}(P_{j}u, P_{k}v)}_{L^{2}}\\
  &\lesssim \sum_{k\le j-3} \norm{P_{j}u}_{L^{2}}\norm{P_{k}v}_{L^{\infty}}\\
  &\lesssim \sum_{k\le j-3} \norm{P_{j}u}_{L^{2}}\,2^{(1-\gamma)k}\Norm{P_{k} \vD^{\gamma-\frac{1}{2}}v}_{L^{6}}\\
  &\lesssim  2^{(1-\gamma)j}\norm{P_{j}u}_{L^{2}}\Norm{ \vD^{\gamma-\frac{1}{2}}v}_{L^{6}}.
\end{align*}
Since $\gamma \ge \frac{1}{2}$, we obtain
\begin{align*}
  \norm{\aD^{\frac{1-\gamma}{2}}\vD^{\gamma}\Theta_{XL}(u,v)}_{L^{2}_{x}} 
  &\lesssim  \RB{\sum_{j\in\Z} \ab{2^{j}}^{-1-\gamma} \norm{P_{j}u}_{L^{2}}^{2}}^{\frac{1}{2}} \Norm{\aD^{\gamma-\frac{1}{2}}v}_{L^{6}} \\
  &\lesssim  \norm{u}_{L^{2}} \Norm{\aD^{\gamma-\frac{1}{2}}v}_{L^{6}}.
\end{align*}
Similarly, we obtain
\begin{align*}
  \norm{\aD^{\frac{1-\gamma}{2}}\vD^{\gamma}\Theta_{LX}(u,v)}_{L^{2}_{x}} 
  &\lesssim  \norm{v}_{L^{2}} \Norm{\aD^{\gamma-\frac{1}{2}}u}_{L^{6}}. \qedhere
\end{align*}
\end{proof}

\subsection{Cubic terms}

\smallskip
\begin{lemma}\label{lem-cubic}
Let $\gamma\in [\frac{1}{2},1]$.
We set $\frac{1}{q(\ve)}=\frac{1}{4}+\frac{\ve}{3}$.
For $\ve>0$,
\begin{align*}
  &\norm{\aD\Omega(\vD^{\gamma}(uv),w)}_{L^{1}_{t}L^{2}_{x}} 
  \lesssim  \norm{u}_{L^{\infty}_{t}L^{2}_{x}}\norm{\aD v}_{L^{2}_{t}L^{6}_{x}}
  \norm{\aD w}_{L^{2}_{t}L^{6}_{x}},\\
  &\norm{\aD\Omega(N, W u)}_{L^{1}_{t}L^{2}_{x}} 
  \lesssim  \norm{N}_{L^{2}_{t}\dot{B}^{-\frac{1}{4}-\ve}_{q(-\ve)}} \norm{W}_{L^{\infty}_{t}L^{2}_{x}}\norm{\aD u}_{L^{2}_{t}L^{6}_{x}},\\
  &\Norm{\aD^{\frac{1-\gamma}{2}}\vD^{\gamma}\Theta(Nu,v)}_{L^{1}_{t}L^{2}_{x}} 
  \lesssim \norm{N}_{L^{\infty}_{t}L^{2}_{x}}\norm{\aD u}_{L^{2}_{t}L^{6}_{x}} \norm{\aD v}_{L^{2}_{t}L^{6}_{x}}.
\end{align*}
\end{lemma}

\begin{proof}
For the first estimate, it suffices to show
\[
  \norm{\aD\Omega(\vD^{\gamma}(uv),w)}_{L^{2}_{x}} 
  \lesssim  \norm{u}_{L^{2}_{x}}\norm{\aD v}_{L^{6}_{x}}
  \norm{\aD w}_{L^{6}_{x}}.\\
\]
We first apply Lemma \ref{HL-interaction} to estimate the HL-type interaction to obtain
\begin{align*}
  \norm{\aD\Omega(\vD^{\gamma}(uv),w)}_{L^{2}_{x}} 
  &= \Norm{\sum_{j\in\Z} \aD\Omega(P_{j}\vD^{\gamma}(uv), P_{\le j-3} w)}_{L^{2}}\\
  &\lesssim \RB{\sum_{j\in \Z} 2^{-2j}
  \norm{\aD\vD \Omega(P_{j} \vD^{\gamma}(uv), P_{\le j-3} w)}_{L^{2}}^{2}}^{\frac{1}{2}}.
\end{align*}
Now we apply Proposition \ref{bilinear-Fourier-multiplier} to obtain
\begin{align*}
  \norm{\aD \vD\Omega(P_{j}\vD^{\gamma}(uv), P_{\le j-3} w)}_{L^{2}}
  &\le  \sum_{k\le j-3} \norm{\aD \vD\Omega(P_{j}\vD^{\gamma}(uv), P_{k} w)}_{L^{2}}\\
  &\lesssim \sum_{k\le j-3} \norm{P_{j}\vD^{\gamma}(uv)}_{L^{2}}\norm{P_{k}w}_{L^{\infty}}\\
  &\lesssim \sum_{k\le j-3} 2^{\gamma j}\norm{P_{j}(uv)}_{L^{2}}\, 
    2^{(1-\gamma)k}\Norm{\vD^{\gamma -\frac{1}{2}} w}_{L^{6}}\\
  &\lesssim 2^{j}\norm{P_{j}(uv)}_{L^{2}} \Norm{\vD^{\gamma-\frac{1}{2}}w}_{L^{6}}.
\end{align*}
Since $\gamma \ge \frac{1}{2}$, we obtain
\[
 \norm{\aD\Omega(\vD^{\gamma}(uv),w)}_{L^{2}_{x}} 
  \lesssim  \norm{uv}_{L^{2}}\Norm{\aD^{\gamma -\frac{1}{2}} w}_{L^{6}}
 \lesssim  \norm{u}_{L^{2}} \norm{v}_{L^{\infty}}\Norm{\aD^{\gamma -\frac{1}{2}} w}_{L^{6}}.
\]
Applying $\norm{v}_{L^{\infty}}\lesssim \norm{\aD v}_{L^{6}}$, we obtain the desired estimate.
%

\smallskip
For the second estimate, it suffices to show
\[
  \norm{\aD\Omega(N, Wu)}_{L^{2}_{x}} 
  \lesssim  \norm{N}_{\dot{B}^{-\frac{1}{4}-\ve}_{q(-\ve)}} \norm{W}_{L^{2}}
  \norm{\aD u}_{L^{6}_{x}}.\\
\]
And this is an immediate consequence of Lemma  \ref{lem-boundary-3}.
In fact, we obtain
\begin{align*}
  \norm{\aD\Omega(N, Wu)}_{L^{2}_{x}} 
  &\lesssim \norm{N}_{\dot{B}^{-\frac{1}{4}-\ve}_{q(-\ve)}} \norm{Wu}_{L^{2}}
  \lesssim \norm{N}_{\dot{B}^{-\frac{1}{4}-\ve}_{q(-\ve)}} \norm{W}_{L^{2}}
    \norm{u}_{L^{\infty}}.
\end{align*}
Applying $\norm{u}_{L^{\infty}}\lesssim \norm{\aD u}_{L^{6}}$, we obtain the desired estimate.

\smallskip
For the third estimate, it suffices to show
\[
  \norm{\aD^{\frac{1-\gamma}{2}}\vD^{\gamma}\Theta(Nu,v)}_{L^{2}_{x}} 
  \lesssim  \norm{N}_{L^{2}} \norm{\aD u}_{L^{6}}
  \norm{\aD v}_{L^{6}_{x}}.
\]
Part of this is an immediate consequence of Lemma  \ref{lem-boundary-3}.
In fact, we obtain
\begin{align*}
  \norm{\aD^{\frac{1-\gamma}{2}}\vD^{\gamma}\Theta_{XL}(Nu,v)}_{L^{2}_{x}} 
  &\lesssim \norm{Nu}_{L^{2}} \norm{\aD v}_{L^{6}}
  \le \norm{N}_{L^{2}}\norm{u}_{L^{\infty}}  \norm{\aD v}_{L^{6}}.
\end{align*}
Applying $\norm{u}_{L^{\infty}}\lesssim \norm{\aD u}_{L^{6}}$, we obtain the desired estimate.

On the other hand, we apply Lemma \ref{HL-interaction} to estimate the LH-type interaction  to obtain
\begin{align*}
  \norm{\aD^{\frac{1-\gamma}{2}}\vD^{\gamma}\Theta_{LX}(Nu,v)}_{L^{2}_{x}} 
  &= \Norm{\sum_{j\in\Z} \aD^{\frac{1-\gamma}{2}}\vD^{\gamma}\Theta_{LX}(P_{\le j-3}(Nu), P_{j}v)}_{L^{2}}\\
  &\lesssim \RB{\sum_{j\in \Z} \ab{2^{j}}^{-1-\gamma}\, 2^{2(\gamma -1)j}
  \norm{\aD\vD \Theta_{LX}(P_{\le j-3}(Nu), P_{j}v)}_{L^{2}}^{2}}^{\frac{1}{2}}.
\end{align*}
Now we apply Proposition \ref{bilinear-Fourier-multiplier} to obtain
\begin{align*}
  \norm{\aD\vD \Theta_{LX}(P_{\le j-3}(Nu), P_{j}v)}_{L^{2}}
  &\le  \sum_{k\le j-3} \norm{\aD\vD \Theta_{LX}(P_{k}(Nu), P_{j}v)}_{L^{2}}\\
  &\lesssim \sum_{k\le j-3} \norm{P_{k}(Nu)}_{L^{3}}\norm{P_{j}v}_{L^{6}}\\
  &\lesssim \sum_{k\le j-3} 2^{k}\norm{P_{k}(Nu)}_{L^{\frac{3}{2}}}\norm{P_{j} v}_{L^{6}}\\
  &\lesssim  2^{j}\norm{Nu}_{L^{\frac{3}{2}}}\norm{P_{j} v}_{L^{6}}.
\end{align*}
Therefore, we obtain
\begin{align*}
  \norm{\aD^{\frac{1-\gamma}{2}}\vD^{\gamma}\Theta_{LX}(Nu,v)}_{L^{2}_{x}} 
  &\lesssim \norm{Nu}_{L^{\frac{3}{2}}}\RB{\sum_{j\in\Z} \ab{2^{j}}^{-1-\gamma}\, 2^{2\gamma j} \norm{P_{j} v}_{L^{6}}^{2}}^{\frac{1}{2}}\\
  &\lesssim \norm{N}_{L^{2}} \norm{u}_{L^{6}} \norm{v}_{L^{6}}.  \qedhere
\end{align*}
\end{proof}

\section{Proof of Theorem \ref{thm-1}}

In this section we give a proof of Theorem \ref{thm-1} by using the contraction mapping principle.
We set
\begin{align*}
  &\norm{u}_{X_{S}}=\norm{\aD u}_{L^{\infty}_{t}L^{2}}+\norm{\aD u}_{L^{2}_{t}\dot{B}^{\frac{1}{4}+\ve}_{q(\ve)}},\\
  &\norm{N}_{X_{W}}=\norm{\aD^{\frac{1-\gamma}{2}} N}_{L^{\infty}_{t}L^{2}}+\norm{\aD^{\frac{1-\gamma}{2}}  N}_{L^{2}_{t}\dot{B}^{-\frac{1}{4}-\ve}_{q(-\ve)}},
\end{align*}
and we define our resolution space by 
\[
  X_{\eta}=\cb{(u,N)\mid \norm{(u,N)}_{X}\le \eta},
\]
where $\norm{(u,N)}_{X}=\norm{u}_{X_{S}}+\norm{N}_{X_{W}}$.

For $u_{0}\in H^{1}$, $N_{0}\in H^{\frac{1-\gamma}{2}}$ with $\norm{u_{0}}_{H^{1}}+\norm{N_{0}}_{H^{\frac{1-\gamma}{2}}}\le \rho$, we define mappings
\begin{equation*}
\begin{aligned}
    \varPhi_{S}[u,N](t)
    &=S(t)u_{0}  + S(t)\Omega(N_{0},u_{0}) - \Omega(N(t), u(t))\\
    &\qquad {} -i \int_{0}^{t} S(t-\tau) (Nu)_{RL+LH+HH}(\tau)\,d\tau\\
    &\qquad {} -i \int_{0}^{t} S(t-\tau) \Omega\Rb{\vD^{\gamma}\rb{u\,\ol{u}}, u}(\tau)\,d\tau
   -i \int_{0}^{t} S(t-\tau) \Omega(N, Nu)(\tau)\,d\tau,
\end{aligned}
\end{equation*}
\begin{equation*}
\begin{aligned}
    \varPhi_{W}[u,N](t)
  &=W(t)N_{0}  + W(t)\vD^{\gamma}\Theta(u_{0}, \ol{u_{0}}) 
    - \vD^{\gamma}\Theta(u(t), \ol{u(t)})\\
    &\qquad {} -i \int_{0}^{t} W(t-\tau)\vD^{\gamma} (u\,\ol{u})_{RL+LR+HH}(\tau)\,d\tau\\
    &\qquad {} -i \int_{0}^{t} W(t-\tau)\vD^{\gamma} \Theta\rb{Nu, \ol{u}}(\tau)\,d\tau
   +i \int_{0}^{t} W(t-\tau) \vD^{\gamma} \Theta(u, \ol{Nu})(\tau)\,d\tau.
\end{aligned}
\end{equation*}

We show that $\varPhi[u,N]=(\varPhi_{S}[u,N],\varPhi_{W}[u,N])$ is a map from $X_{\eta}$ to itself and is a contraction mapping if $\eta =2C_{0}\rho$ is sufficiently small.

For $(u,N)\in X_{\eta}$, applying Proposition \ref{radial-Strichartz} we obtain
\begin{align*}
  \norm{\varPhi_{S}[u,N]}_{X_{S}}
  &\le C_{0}\norm{u_{0}}_{H^{1}} 
    + C \norm{\aD\Omega(N_{0},u_{0})}_{L^{2}}
    + C \norm{\aD\Omega(N, u)}_{L^{\infty}_{t}L^{2}_{x}\cap L^{2}_{t}\dot{B}^{\frac{1}{4}+\ve}_{q(\ve)}}\\
   &\qquad  {}+C \norm{\aD (Nu)_{LH+HH}}_{L^{1}_{t}L^{2}}
     + C\norm{\aD (Nu)_{RL}}_{L^{a'}_{t}\dot{B}^{\frac{3}{2}-\frac{2}{a}-\frac{3}{b}}_{b'}}\\
   &{}\qquad +C \norm{\aD\Omega(\vD^{\gamma}(u\ol{u}),u)}_{L^{1}_{t}L^{2}_{x}} 
   +C\norm{\aD\Omega(N, N u)}_{L^{1}_{t}L^{2}_{x}}.
\end{align*}
Then, applying Lemmas \ref{lem-boundary-1}, \ref{lem-boundary-2}, \ref{lem-quadratic-1}, \ref{lem-cubic}, we obtain
\begin{align*}
  \norm{\varPhi_{S}[u,N]}_{X_{S}}
  &\le C_{0}\norm{u_{0}}_{H^{1}} + C\norm{N_{0}}_{L^{2}}\norm{u_{0}}_{H^{1}}\\
  &\qquad {}+ C\norm{N}_{L^{\infty}L^{2}}\norm{u}_{L^{\infty}H^{1}}
    + C\norm{N}_{L^{\infty}L^{2}}\norm{\aD u}_{L^{2}\dot{B}^{\frac{1}{4}+\ve}_{q(\ve)}}\\
  &\qquad {}+ C\norm{N}_{L^{2}_{t}\dot{B}^{-\frac{1}{4}-\ve}_{q(-\ve)}} 
    \norm{\aD u}_{L^{2}_{t}\dot{B}^{\frac{1}{4}+\ve}_{q(\ve)}}
    + C \norm{N}_{L^{2}_{t}\dot{B}^{-\frac{1}{4}-\ve}_{q(-\ve)}} 
    \norm{u}_{L^{\infty}_{t}L^{2}\cap L^{2}_{t}\dot{B}^{\frac{1}{4}+\ve}_{q(\ve)}}\\
  &\qquad {}+ C \norm{u}_{L^{\infty}_{t}L^{2}_{x}}\norm{\aD u}_{L^{2}_{t}L^{6}_{x}}^{2}
   + C \norm{N}_{L^{2}_{t}\dot{B}^{-\frac{1}{4}-\ve}_{q(-\ve)}} \norm{N}_{L^{\infty}_{t}L^{2}_{x}}\norm{\aD u}_{L^{2}_{t}L^{6}_{x}}\\
  &\le C_{0} \norm{u_{0}}_{H^{1}} + C\cb{ \rho^{2} +  \norm{N}_{X_{W}}\norm{u}_{X_{S}}
    + \norm{u}_{X_{S}}^{3} + \norm{N}_{X_{W}}^{2}\norm{u}_{X_{S}}}.
\end{align*}
Similarly, we have
\begin{align*}
  \norm{\varPhi_{W}[u,N]}_{X_{W}}
  &\le C_{0}\norm{N_{0}}_{H^{\frac{1-\gamma}{2}}} 
  + C \norm{\aD^{\frac{\gamma -1}{2}}\vD^{\gamma}\Theta(u_{0},\ol{u_{0}})}_{L^{2}}\\
   &\qquad  {} + C \norm{\aD^{\frac{\gamma -1}{2}}\vD^{\gamma}\Theta(u,\ol{u})}_{L^{\infty}_{t}L^{2}_{x}\cap L^{2}_{t}\dot{B}^{\frac{1}{4}+\ve}_{q(\ve)}}\\
   &\qquad  {}+C \norm{\aD^{\frac{1-\gamma}{2}}\vD^{\gamma} (u \ol{u})_{HH}}_{L^{1}_{t}L^{2}_{x}}
     + C\norm{\aD^{\frac{1-\gamma}{2}}\vD^{\gamma} (uv)_{RL+LR}}_{L^{a'}_{t}\dot{B}^{\frac{3}{2}-\frac{1}{a}-\frac{3}{b}}_{b'}}\\
   &{}\qquad +C\Norm{\aD^{\frac{\gamma -1}{2}}\vD^{\gamma}\Theta(Nu,\ol{u})}_{L^{1}_{t}L^{2}_{x}}\\
  &\le C_{0}\norm{N_{0}}_{H^{\frac{1-\gamma}{2}}} 
     + C\norm{u_{0}}_{H^{1}}^{2}  + C\norm{u}_{L^{\infty}H^{1}}^{2}
     +  C\norm{u}_{L^{2}_{t}L^{6}_{x}} \norm{u}_{L^{\infty}_{t}L^{2}_{x}}\\
  &\qquad {}+ C\norm{u}_{L^{2}_{t}\dot{B}^{\frac{1}{4}+\ve}_{q(\ve)}} \norm{\aD^{\frac{\gamma}{2}} u}_{L^{2}_{t}\dot{B}^{\frac{1}{4}+\ve}_{q(\ve)}}
    + C  \norm{u}_{L^{\infty}_{t}L^{2}\cap L^{2}_{t}\dot{B}^{\frac{1}{4}+\ve}_{q(\ve)}}^{2}\\
  &\qquad {}+ C \norm{N}_{L^{\infty}_{t}L^{2}_{x}}\norm{\aD u}_{L^{2}_{t}L^{6}_{x}} \norm{\aD u}_{L^{2}_{t}L^{6}_{x}}\\
  &\le C_{0}\norm{N_{0}}_{H^{\frac{1-\gamma}{2}}} 
    + C\cb{ \rho^{2} +  \norm{u}_{X_{S}}^{2}
    + \norm{u}_{X_{S}}^{3} + \norm{N}_{X_{W}}\norm{u}_{X_{S}}}^{2}.
\end{align*}
Therefore, we obtain
\[
  \norm{\varPhi[u,N]}_{X}=\norm{\varPhi_{S}[u,N]}_{X_{S}}
    + \norm{\varPhi_{W}[u,N]}_{X_{W}}
    \le C_{0}\rho + C_{1}\cb{\rho^{2}+\eta^{2}+\eta^{3}}.
\]
Then, by setting $\eta=2C_{0}\rho$, we conclude that
\[
  \norm{\varPhi[u,N]}_{X}\le C_{0}\rho + C_{1}\cb{\rho + (2C_{0})^{2}\rho
    + + (2C_{0})^{3}\rho^{2}}\rho \le \eta,
\]
provided that $C_{1}\cb{\rho + (2C_{0})^{2}\rho
    +  (2C_{0})^{3}\rho^{2}}\le C_{0}$.

We can also prove that $\varPhi$ is a contraction map in a usual manner.

\medskip
We next show that the solution satisfy \eqref{cond-conti}.
The integral equations are written as
\begin{equation}\label{int-normal-schrodinger-2}
\begin{aligned}
  u(t)&=S(t)u_{0}  + S(t)\Omega(N_{0},u_{0}) - \Omega(N(t), u(t))\\
    &\qquad {} -i \int_{0}^{t} S(t-\tau) F(\tau)\,d\tau
     -i \int_{0}^{t} S(t-\tau) (Nu)_{RL}(\tau)\,d\tau,
\end{aligned}
\end{equation}
\begin{equation}\label{int-normal-wave-2}
\begin{aligned}
  N(t)&=W(t)N_{0}  + W(t)\vD^{\gamma}\Theta(u_{0}, \ol{u_{0}}) 
    - \vD^{\gamma}\Theta(u(t), \ol{u(t)})\\
    &\qquad {} -i \int_{0}^{t} W(t-\tau)G(\tau)\,d\tau
    -i \int_{0}^{t} W(t-\tau)\vD^{\gamma} (u\,\ol{u})_{RL+LR}(\tau)\,d\tau,
\end{aligned}
\end{equation}
where $u_{0}$, $\Omega(N_{0},u_{0})\in H^{1}$, $N_{0}$, $\vD^{\gamma}\Theta(u_{0}, \ol{u_{0}})\in H^{\frac{1-\gamma}{2}}$, and
\begin{align*}
  &F=(Nu)_{LH+HH}+\Omega\Rb{\vD^{\gamma}\rb{u\,\ol{u}}, u}+\Omega(N, Nu)\in L^{1}_{t}H^{1}_{x},\\
  &G=(u\,\ol{u})_{HH}+\vD^{\gamma} \Theta\rb{Nu, \ol{u}}+ \vD^{\gamma} \Theta(u, \ol{Nu})\in L^{1}_{t}H^{\frac{1-\gamma}{2}}_{x}.
\end{align*}
So, to show \eqref{cond-conti}, it suffices to check the third and the fifth term on the right hand side in each of the integral equations.
As for the third term of the first equation, applying Lemma \ref{lem-boundary-1}, we obtain
\begin{align*}
  &\norm{\Omega(N(t), u(t))-\Omega(N(t'), u(t'))}_{H^{1}}\\
  &\le \norm{\Omega(N(t)-N(t'),u(t))}_{H^{1}}+ \norm{\Omega(N(t'),u(t)-u(t'))}_{H^{1}}\\
  &\lesssim \eta\cb{ \norm{N(t)-N(t')}_{L^{2}}+ \norm{u(t)-u(t')}_{H^{1}}},
\end{align*}
and this term can be absorbed to the left hand side when we estimate
\[
  \norm{u(t)-u(t')}_{H^{1}}+\norm{N(t)-N(t')}_{H^{\frac{1-\gamma}{2}}}.
\]
As for the fifth term of the first equation, for $t>t'$, applying Proposition \ref{radial-Strichartz} and Lemma \ref{lem-quadratic-1}, we obtain
\begin{equation}\label{integral-resonant-est}
\begin{aligned}
  \NM{\int_{t'}^{t} S(t-\tau) (Nu)_{RL}(\tau)\,d\tau}_{H^{1}}
  &= \NM{\int_{0}^{\rho} S(\rho -\tau)\chi_{(t',t)}(\tau) (Nu)_{RL}(\tau)\,d\tau}_{H^{1}}\\
  &\lesssim \norm{\chi_{(t',t)} \aD (Nu)_{RL}}_{L^{a'}_{\tau}\dot{B}^{\frac{3}{2}-\frac{2}{a}-\frac{3}{b}}_{b'}}\\
  &\lesssim \norm{\chi_{(t',t)} (Nu)_{RL}}_{L^{a'}_{\tau}L^{b'}},
\end{aligned}
\end{equation}
where we take $\rho>t$ in the first place.
Thus, continuity follows from the Lebesgue dominated convergence theorem.
The third and the fifth term of the second equation can be treated similarly.

\medskip
We finally show that the solution is asymptotically free assuming $\gamma\in [\frac{1}{2},1)$.
To prove this, it suffices to show that
\[
  f(t)=S(-t)u(t),\quad g(t)=W(-t)N(t)
\]
have the limit as $t\to\infty$.
The integral equations which $f(t)$, $g(t)$ satisfy are
\begin{equation*}
\begin{aligned}
  f(t)&=u_{0}  + \Omega(N_{0},u_{0}) - S(-t)\Omega(N(t), u(t))\\
    &\qquad {} -i \int_{0}^{t} S(-\tau) F(\tau)\,d\tau
     -i \int_{0}^{t} S(-\tau) (Nu)_{RL}(\tau)\,d\tau,
\end{aligned}
\end{equation*}
\begin{equation*}
\begin{aligned}
  g(t)&=N_{0}  + \vD^{\gamma}\Theta(u_{0}, \ol{u_{0}}) 
    - W(-t)\vD^{\gamma}\Theta(u(t), \ol{u(t)})\\
    &\qquad {} -i \int_{0}^{t} W(-\tau)G(\tau)\,d\tau
    -i \int_{0}^{t} W(-\tau)\vD^{\gamma} (u\,\ol{u})_{RL+LR}(\tau)\,d\tau.
\end{aligned}
\end{equation*}
Since $F\in L^{1}_{t}H^{1}$, we have
\[
  \NM{\int_{t'}^{t} S(-\tau) F(\tau)\,d\tau}_{H^{1}}
  \le \ABS{\int_{t'}^{t} \norm{F(\tau)}_{H^{1}}}
  \to 0,\quad t,t'\to\infty.
\]
We also obtain
\begin{equation*}
\begin{aligned}
  \NM{\int_{t'}^{t} S(-\tau) (Nu)_{RL}(\tau)\,d\tau}_{H^{1}}
  &= \NM{\int_{0}^{\rho} S(\rho -\tau)\chi_{(t',t)}(\tau) (Nu)_{RL}(\tau)\,d\tau}_{H^{1}}\\
  &\lesssim \norm{\chi_{(t',t)} (Nu)_{RL}}_{L^{a'}_{\tau}L^{b'}}\to 0,\quad t,t'\to\infty.
\end{aligned}
\end{equation*}
The situation is similar in the second equation.
Thus, to prove 
\[
  \norm{f(t)-f(t')}_{H^{1}}\to 0,\qquad  
  \norm{g(t)-g(t')}_{H^{\frac{1-\gamma}{2}}}\to 0,\qquad t,t'\to\infty,
\] 
it suffices to show
\[
  \lim_{t\to\infty}\norm{\Omega(N(t), u(t))}_{H^{1}}=0,\qquad
  \lim_{t\to\infty}\Norm{\vD^{\gamma}\Theta(u(t), \ol{u(t)})}_{H^{\frac{1-\gamma}{2}}}=0.
\]
To prove this, since we have
\[
  \norm{\Omega (N(t),u(t))}_{H^{1}}\in L^{2}(0,\infty),\qquad
  \norm{\vD^{\gamma}\Theta (u(t),\overline{u(t)})}_{H^{\frac{1-\gamma}{2}}}\in L^{2}(0,\infty)
\]
by Lemma \ref{lem-boundary-3}, it suffices to prove
\[
  \Omega (N,u)\in BUC([0,\infty); H^{1}), \qquad 
  \vD^{\gamma}\Theta (u,\overline{u})\in BUC([0,\infty); H^{\frac{1-\gamma}{2}}),
\]
where $BUC([0,\infty); X)$ denotes the space of a Banach space $X$-valued bounded and uniformly continuous functions.
And this is an immediate consequence of
\begin{equation}\label{buc}
  u\in BUC([0,\infty); H^{1-\delta}), \quad N\in BUC([0,\infty); H^{\frac{1-\gamma}{2}-\delta})
\end{equation}
with $\delta\in (0,\frac{1-\gamma}{2})$, since by Lemma \ref{lem-boundary-1} we have
\begin{align*}
  &\norm{\Omega(N(t), u(t))- \Omega(N(t'), u(t'))}_{H^{1}}\\
  & \qquad \lesssim \norm{N(t)-N(t')}_{L^{2}} \norm{u(t)}_{H^{1}}
    +  \norm{N(t)}_{L^{2}} \norm{u(t)-u(t')}_{H^{1-\delta}},\\
  &\Norm{\vD^{\gamma}\Theta (u(t),\overline{u(t)})
    - \vD^{\gamma}\Theta (u(t'),\overline{u(t')})}_{H^{\frac{1-\gamma}{2}}}
    \lesssim \norm{u}_{L^{\infty}_{t}H^{1}}\norm{u(t)-u(t')}_{H^{1-\delta}}.
\end{align*}

\eqref{buc} is proved based on estimates
\[
  \norm{S(t)f -S(t')f}_{H^{1-\delta}}
  \le C\abs{t-t'}^{\frac{\delta}{2}} \norm{f}_{H^{1}},
\]
and for $\wt{F}=F+(Nu)_{RL}$, $t'<t$,
\begin{align*}
  &\NM{\int_{0}^{t}S(t-\tau) \wt{F}(\tau)\,ds -\int_{0}^{t'} S(t'-\tau) \wt{F}(\tau)\,d\tau}_{H^{1-\delta}}\\
  & \le \NM{\int_{t'}^{t}S(t-\tau) \wt{F}(\tau)\,d\tau}_{H^{1-\delta}}
    +\NM{\int_{0}^{t'}\Cb{S(t-\tau) F(\tau)-S(t'-\tau)F(\tau)}\,d\tau}_{H^{1-\delta}}\\
  & = \NM{\int_{0}^{\rho}  S(\rho -\tau) \chi_{(t',t)}(\tau) \wt{F}(\tau)\,d\tau}_{H^{1-\delta}}
    +\NM{\cb{S(t-t')-I} \int_{0}^{t'}S(t'-\tau) \wt{F}(\tau)\,d\tau}_{H^{1-\delta}}\\
  & \lesssim \norm{\chi_{(t',t)}F}_{L^{1}H^{1}} 
    + \norm{\chi_{(t',t)}  (Nu)_{RL}}_{L^{a'}_{\tau}L^{b'}}
    + \abs{t-t'}^{\frac{\delta}{2}}\,  \NM{\int_{0}^{t'}S(t'-\tau) \wt{F}(\tau)\,d\tau}_{L^{\infty}_{t'}H^{1}},
\end{align*}
where $t<\rho$.
As for the boundary term of the integral equation \eqref{int-normal-schrodinger-2}, applying Lemma \ref{lem-boundary-1} we estimate
\begin{align*}
  &\norm{\Omega(N(t), u(t))- \Omega(N(t'), u(t'))}_{H^{1-\delta}}\\
  &\le \norm{\Omega\rb{N(t)-N(t'), u(t)}}_{H^{1-\delta}}
    + \norm{\Omega\rb{N(t'), u(t)-u(t')}}_{H^{1-\delta}}\\
  & \le C \norm{N(t)-N(t')}_{L^{2}} \norm{u(t)}_{H^{1}}
    + C\norm{N(t)}_{L^{2}} \norm{u(t)-u(t')}_{H^{1-\delta}},
  &
\end{align*}
and these terms can be absorbed to the left hand side when we estimate
\[
  \norm{u(t)-u(t')}_{H^{1-\delta}}+\norm{N(t)-N(t')}_{H^{{\frac{1-\gamma}{2}-\delta}}},
\]
since $\norm{u(t)}_{H^{1}}$, $\norm{N(t)}_{L^{2}}\le \eta$.
Finally, we notice that we use
\[
  \norm{W(t) g -W(t') g }_{H^{{\frac{1-\gamma}{2}-\delta}}}
  \le C\abs{t-t'}^{\delta} \norm{g}_{H^{{\frac{1-\gamma}{2}}}}.
\]
to handle \eqref{int-normal-wave-2}.

\section{Proof of Theorem \ref{thm-2}}

We first show that for any $T>0$ we have
\begin{equation}\label{proof-thm2-ac}
  u\in AC\rb{(0,T); H^{-1}},\quad N\in AC\rb{(0,T); H^{-\frac{1+\gamma}{2}}},
\end{equation}
where $AC((0,T); X)$ denotes the space of a Banach space $X$-valued absolutely continuous functions.

To prove this, we first consider estimate of each terms on the right hand side of \eqref{int-normal-schrodinger-2}.

For $f\in H^{1}$, we have
\[
  S(t)f\in C^{1}((0,\infty); H^{-1}).
\]
and
\[
  \norm{S(t+h)f- S(t)f}_{H^{-1}}\le \abs{h}\norm{f}_{H^{1}}.
\]
Since $F\in L^{1}_{t}H^{1}_{x}$, we see that
\[
  \int_{0}^{t} S(t-\tau)F(\tau)\,d\tau
  =S(t)\int_{0}^{t} S(-\tau)F(\tau)\,d\tau\in AC([0,\infty); H^{-1}).
\]
Similarly, for any $T>0$ we have
\[
  \int_{0}^{t} S(t-\tau)(Nu)_{RL}(\tau)\,d\tau \in AC((0,T); H^{-1}),
\]
since
\[
  \norm{(Nu)_{RL}}_{L^{1}(0,T;H^{1})}
  \lesssim T^{\frac{1}{a}} \norm{(Nu)_{RL}}_{L^{a'}_{t}L^{b'}_{x}}<\infty,
\]
where $(a,b)$ is radial Schr\"{o}dinger admissible.
Below we set these two integral terms $F_{S}(t)$.

And applying Lemma \ref{lem-boundary-1}, we obtain
\begin{align*}
  & \norm{\Omega\rb{N(t+h),u(t+h)} -\Omega\rb{N(t),u(t)}}_{H^{-1}}\\
  & \lesssim \eta \cb{ \norm{N(t+h)-N(t)}_{H^{-\frac{1+\gamma}{2}}}
    + \norm{u(t+h)-u(t)}_{H^{-1}} }.
\end{align*}

Now for $\forall M$, we set 
\[
  0\le a_{1}<b_{1}<a_{2}<b_{2}<\cdots <a_{M}<b_{M}\le T
\]
and set $\dsp I_{M}=\bigcup_{j=1}^{M}(a_{j},b_{j})$.
Then,
\begin{align*}
  \sum_{j=1}^{M}\, \norm{u(b_{j})-u(a_{j})}_{H^{-1}}
  & \le \abs{I_{M}}\, \norm{u_{0}}_{H^{1}} 
    + C \abs{I_{M}}\, \norm{N_{0}}_{L^{2}}\norm{u_{0}}_{H^{1}}\\
  & \qquad + C\eta \sum_{j=1}^{M}\Cb{ \Norm{N(b_{j})-N(a_{j})}_{H^{-\frac{1+\gamma}{2}}}
    + \Norm{u(b_{j})-u(a_{j})}_{H^{-1}} }\\
  & \qquad + \sum_{j=1}^{M}\norm{F_{S}(b_{j})-F_{S}(a_{j})}_{H^{-1}}.
\end{align*}
Similarly, by setting
\[
  G_{W}(t)=\int_{0}^{t} W(-\tau)G(\tau)\,d\tau
    +\int_{0}^{t} W(-\tau)\vD^{\gamma} (u\,\ol{u})_{RL+LR}(\tau)\,d\tau
    \in AC([0,T]; H^{-\frac{1+\gamma}{2}}),
\]
we have
\begin{align*}
  &\sum_{j=1}^{M}\, \norm{N(b_{j})-N(a_{j})}_{H^{-\frac{1+\gamma}{2}}}
   \le \abs{I_{M}}\, \norm{N_{0}}_{H^{\frac{1-\gamma}{2}}} 
    + C \abs{I_{M}}\, \norm{u_{0}}_{H^{1}}^{2}\\
  & \qquad +C\eta \sum_{j=1}^{M} \norm{u(b_{j})-u(a_{j})}_{H^{-1}}
    +  \sum_{j=1}^{M}\norm{G_{W}(b_{j})-G_{W}(a_{j})}_{H^{-\frac{1+\gamma}{2}}}.
\end{align*}
Therefore, for $\forall \ve>0$, taking $\abs{I_{M}}\ll 1$, we conclude that
\[
  \sum_{j=1}^{M}\cb{ \norm{u(b_{j})-u(a_{j})}_{H^{-1}} +
    \norm{N(b_{j})-N(a_{j})}_{H^{-\frac{1+\gamma}{2}}}}< \ve.
\]

\medskip
Finally, we prove \eqref{thm-2-c1} and $(u,N)$ satisfy the original system \eqref{differential-equation}.
By \eqref{proof-thm2-ac}, $u$, $N$ are differentiable a.e.~and we set
\[
  R_{S}=i\d_{t}u-\Delta u-Nu,\qquad
  R_{W}=i\d_{t}N-\abs{\nabla}N-\vD^{\gamma} \abs{u}^{2}.
\]
Then, by using the integral equations \eqref{int-normal-schrodinger}, \eqref{int-normal-wave}, we see that
\begin{align*}
  & R_{S}=\Omega(R_{W}, u) + \Omega(N, R_{S})\quad \text{in}\ H^{-1}, \quad \text{a.e.}\ t,\\
  & R_{W}=\Theta(R_{S}, \ol{u}) + \Theta(u, \ol{R_{S}})\quad \text{in}\ H^{-\frac{1+\gamma}{2}}, \quad \text{a.e.}\ t
\end{align*}
hold.
More precisely, we carry out inverse procedure to derive \eqref{int-normal-schrodinger}, \eqref{int-normal-wave}.
Then, since $\norm{u(t)}_{H^{1}}\ll 1$, $\norm{N(t)}_{H^{\frac{1-\gamma}{2}}}\ll 1$, we obtain
\[
  R_{S}=0\quad\text{in}\ H^{-1}, \qquad R_{W}=0 \quad \text{in}\ H^{-\frac{\gamma+1}{2}}, \quad \text{a.e.}\ t.
\]
From this, we see that $(u,N)$ satisfy the integral equations \eqref{int-normal-schrodinger}, \eqref{int-normal-wave}, which implies \eqref{thm-2-c1}.

\section{Proof of scattering in the case $\gamma=1$.}\label{sec-appendix}

In this section, we give a proof of the scattering of the solution derived in Theorem \ref{thm-1} in the case $\gamma =1$.
In this case, we have a solution 
\begin{equation*}
  u\in C\rb{[0,\infty);H^{1}},\quad N\in C\rb{[0,\infty); L^{2}}
\end{equation*}
to  \eqref{int-normal-schrodinger}, \eqref{int-normal-wave} satisfying $\norm{u}_{X_{S}}+\norm{N}_{X_{W}}\le \eta$, where
\begin{align*}
  &\norm{u}_{X_{S}}=\norm{\aD u}_{L^{\infty}_{t}L^{2}}+\norm{\aD u}_{L^{2}_{t}\dot{B}^{\frac{1}{4}+\ve}_{q(\ve)}},\\
  &\norm{N}_{X_{W}}=\norm{N}_{L^{\infty}_{t}L^{2}}+\norm{N}_{L^{2}_{t}\dot{B}^{-\frac{1}{4}-\ve}_{q(-\ve)}}.
\end{align*}
As in the proof of Theorem \ref{thm-1}, to prove the scattering, it suffices to show
\begin{equation}\label{appendix-boundary-limit}
  \lim_{t\to\infty}\norm{\Omega(N(t), u(t))}_{H^{1}}=0,\qquad
  \lim_{t\to\infty}\Norm{\vD \Theta(u(t), \ol{u(t)})}_{L^{2}}=0.
\end{equation}
To prove this, since we have
\[
  \norm{\Omega (N(t),u(t))}_{H^{1}}\in L^{2}(0,\infty),\qquad
  \norm{\vD \Theta (u(t),\overline{u(t)})}_{L^{2}}\in L^{2}(0,\infty)
\]
by Lemma \ref{lem-boundary-3}, it suffices to prove
\begin{equation}\label{appendix-boundary-buc}
  \Omega (N,u)\in BUC([0,\infty); H^{1}), \qquad 
  \vD \Theta (u,\overline{u})\in BUC([0,\infty); L^{2}).
\end{equation}
To prove this, we prepare the following lemma, which is a slight modification of Lemma \ref{lem-boundary-1}.

\begin{lemma}\label{lem-boundary-4}
For $\delta\in (0,\frac{1}{2})$, $s>\frac{1}{2}+\delta$,
\begin{align*}
  \norm{\Omega(N,u)}_{H^{1}} \lesssim \norm{N}_{H^{-\delta}}
  \norm{u}_{H^{s}}.
\end{align*}
\end{lemma}

\begin{proof}
The proof is similar to the one of Lemma \ref{lem-boundary-1}.
Actually, instead of \eqref{pf-lem-boundary-1_1}, we estimate
\begin{align*}
  \ABS{\frac{\ab{\xi}}{\omega(\xi-\eta,\eta)}}\,
  \cP_{XL}(\xi-\eta,\eta)
  &\lesssim  \frac{1}{\abs{\xi-\eta}}\,\cP_{XL}(\xi-\eta,\eta)\\
  &=\frac{1}{\ab{\xi-\eta}^{\delta}} 
    \RB{\frac{\ab{\xi-\eta}}{\abs{\xi-\eta}}}^{\delta}
    \frac{1}{\abs{\xi-\eta}^{1-\delta}}\,\cP_{XL}(\xi-\eta,\eta)\\
  &\lesssim \frac{1}{\ab{\xi-\eta}^{\delta}}
    \frac{\ab{\eta}^{\delta}}{\abs{\eta}},
\end{align*}
where we have used the fact that $(\xi-\eta,\eta)\in \supp\cP_{XL}$ implies $\abs{\eta}\ll \abs{\xi-\eta}\simeq \abs{\xi}$, and that  $f(r)=\Rb{\frac{\ab{r}}{\abs{r}}}^{\delta}=\Ab{\frac{1}{r}}^{\delta}$ is monotone decreasing in $r\in [0,\infty)$.
Thus,
\begin{align*}
  \norm{\aD\Omega(N,u)}_{L^{2}}
  &\lesssim \NM{\int \ab{\xi-\eta}^{-\delta}\abs{\wh{N}(\xi-\eta)}
   \ab{\eta}^{\delta}\,\abs{\eta}^{-1}\abs{\wh{u}(\eta)}\,d\eta}_{L^{2}}\\
  &\le \Norm{\ab{\eta}^{-\delta}\,\wh{N}}_{L^{2}}
  \Norm{\ab{\eta}^{\delta}\abs{\eta}^{-1}\abs{\wh{u}(\eta)}}_{L^{1}}\\
  &\le \Norm{\abs{\eta}^{-1}\,\ab{\eta}^{-s+\delta}}_{L^{2}}
    \norm{N}_{H^{-\delta}}\Norm{\ab{\eta}^{s}\wh{u}(\eta)}_{L^{2}}.
\end{align*}
Since $s >\frac{1}{2}+\delta$ implies $2(-1-s+\delta)<-3$, we obtain the desired estimate.
\end{proof}

Applying Lemma \ref{lem-boundary-4}, for $\delta\in (0,\frac{1}{4})$ we obtain
\begin{align}
  &\norm{\Omega (N(t),u(t)) -\Omega (N(t'),u(t'))}_{H^{1}} \notag\\
  &\lesssim \norm{N(t)-N(t')}_{H^{-\delta}} \norm{u(t)}_{H^{1-\delta}}
    +\norm{N(t')}_{H^{-\delta}} \norm{u(t)-u(t')}_{H^{1-\delta}}\notag\\
  &\lesssim \eta \norm{N(t)-N(t')}_{H^{-\delta}} 
  + \eta \norm{u(t)-u(t')}_{H^{1-\delta}}. \label{appendix-boundary-diff-est}
\end{align}
Similarly, by Lemma \ref{lem-boundary-1} we have
\[
  \Norm{\vD\Theta (u(t),\ol{u(t)}) -\vD\Theta (u(t'), \ol{u(t')})}_{L^{2}}
   \lesssim \eta \norm{u(t)-u(t')}_{H^{1-\delta}}.
\]
Therefore, \eqref{appendix-boundary-buc} follows from
\begin{equation}\label{appendix-buc}
  u\in BUC([0,\infty); H^{1-\delta}), \quad 
  N\in BUC([0,\infty); H^{-\delta}).
\end{equation}
The proof of \eqref{appendix-buc} is similar to the case $\gamma\in [\frac{1}{2},1)$, except for the use of \eqref{appendix-boundary-diff-est} to estimate the boundary term.


\begin{thebibliography}{99}

\bibitem{beck-pusateri-sosoe-wong}
T. Beck, F. Pusateri, P. Sosoe, P. Wong,
On global solutions of a Zakharov type system,
Nonlinearity \textbf{28} (2015), 3419--3441.

\bibitem{germain-masmoudi-shatah}
P. Germain, N. Masmoudi, J. Shatah,
Global solutions for 3D quadratic Schr\"{o}dinger equations,
Int. Math. Res. Not. (2009), 706--728.




\bibitem{guo-nakanishi_1}
Z. Guo, K. Nakanishi,
Small energy scattering for the Zakharov system with radial symmetry,
IMRN (2014), No.\ 9, 2327—2342.

\bibitem{guo-wang_2014}
Z. Guo, Y. Wang,
Improved Strichartz estimates for a class of dispersive equations in the radial case and their applications to nonlinear Schr\"{o}dinger and wave equations,
J. D'Anal. Math. \textbf{124} (2014), 1--38.

\bibitem{hani-pusateri-shatah}
Z. Hani, F. Pusateri, J. Shatah,
Scattering for the Zakharov system in 3 dimensions,
Comm. Math. Phys. \textbf{322} (2013), 731--753. 


\end{thebibliography}
\end{document}